\documentclass[11pt]{amsart}
\usepackage{fullpage}

\usepackage{amssymb,amsmath,amsthm,amscd,mathrsfs,graphicx, color, mathtools, enumerate}
\usepackage[cmtip,all,matrix,arrow,tips,curve]{xy}
\usepackage{hyperref}
\usepackage[usenames,dvipsnames]{xcolor}
\usepackage{xypic}
\hypersetup{colorlinks=true,citecolor=ForestGreen,linkcolor=Maroon,urlcolor=NavyBlue}

\numberwithin{equation}{section}
\newtheorem{teo}{Theorem}[section]
\newtheorem{pro}[teo]{Proposition}
\newtheorem{lem}[teo]{Lemma}
\newtheorem{cor}[teo]{Corollary}

\newtheorem{teoalpha}{Theorem}

\theoremstyle{definition}
\newtheorem{dfn}[teo]{Definition}
\newtheorem{exa}[teo]{Example}

\theoremstyle{remark}
\newtheorem{rem}[teo]{Remark}

\global\let\hom\undefined
\DeclareMathOperator{\hom}{Hom}

\newcommand{\til}[1]{{\widetilde{#1}}}

\def\cross{\times}
\def\inv{^{-1}}

\def\cx{{\mathbb C}}
\def\ff{{\mathbb F}}

\def\rat{{\mathbb Q}}

\def\idp{{\mathfrak p}}

\def\iso{\cong}
\def\inject{\hookrightarrow}
\renewcommand{\bar}[1]{{\overline{#1}}}

\DeclareMathOperator{\gal}{Gal}

\DeclareMathOperator{\spec}{Spec}
\DeclareMathOperator{\pic}{Pic}
\DeclareMathOperator{\aj}{AJ}

\DeclareMathOperator{\chow}{CH}

\def\tensor{\otimes}
\newcommand{\st}[1]{\left\{#1\right\}}

\def\kbar{{\bar K}}

\title{Parameter spaces for algebraic equivalence}

\author{Jeffrey D. Achter}
\address{Colorado State University, Department of Mathematics,
Fort Collins, CO 80523,
USA}
\email{j.achter@colostate.edu}

\author{Sebastian Casalaina-Martin }
\address{University of Colorado, Department of Mathematics, 
Boulder, CO 80309, USA }
\email{casa@math.colorado.edu}

\author{Charles Vial}
\address{Universit\"at Bielefeld, Fakult\"at f\"ur Mathematik, 33501 Bielefeld, 
	Germany}
\email{vial@math.uni-bielefeld.de}


\thanks{The first author was partially supported by  grants from the
  Simons Foundation (204164) and the NSA (H98230-14-1-0161, H98230-15-1-0247 and H98230-16-1-0046).  The second author
was partially supported by NSF grant DMS-1101333, a Simons Foundation
Collaboration Grant for Mathematicians
(317572), and NSA grant H98230-16-1-0053.  The third author was supported  by EPSRC Early Career Fellowship
EP/K005545/1.}

\date{\today}

\begin{document}

\begin{abstract}  
A cycle is algebraically trivial if it can be exhibited as the
difference of two fibers in a family of cycles parameterized by a
smooth integral scheme.  Over an algebraically closed field, it is a result of
Weil 
 that it suffices to consider families of cycles parameterized by curves, or by abelian varieties.   In this paper, we extend these results to arbitrary base fields.
The strengthening of these results   turns out to be a key step in our work elsewhere extending Murre's
results  on algebraic representatives for varieties over
algebraically closed fields to arbitrary perfect fields. 
\end{abstract}

\maketitle

\section{Introduction}

Consider a scheme $X$ of finite type over a field $K$. Given a
non-negative integer $i$, the group of dimension-$i$ cycles, denoted
$\operatorname{Z}_i(X)$, is the free abelian group generated by closed
integral subschemes of $X$ of dimension $i$. The Chow group
$\operatorname{CH}_i(X)$ is the quotient of $\operatorname{Z}_i(X)$ by
rational equivalence.  A dimension-$i$ cycle class $a\in
\operatorname{CH}_i(X)$ on $X$ is called algebraically trivial if
there exist a twice-$K$-pointed smooth integral scheme $(T, t_0, t_1)$
of finite type over $K$, and a cycle class $Z\in
\operatorname{CH}_{i+\dim_K T}(T\cross_K X)$, such that $a =
Z_{t_1}-Z_{t_0}$ in $\operatorname{CH}_i(X)$.  Here $Z_{t_j}$,
$j=0,1$, is the refined Gysin fiber of the cycle class $Z$ \cite[\S
6.2]{fulton}.  Algebraic equivalence defines an equivalence relation
on $Z_i(X)$\,; in fact algebraic equivalence is an adequate
equivalence relation (e.g., \cite{samuelequivalence}, \cite[\S 10.3,
Prop.~10.3]{fulton}, \cite[\S 1.1]{murre83}).

Over an algebraically closed field $K=\bar K$, Weil showed (\cite[Lem.~9]{weil54}\,; see also \cite[\S III.1]{Lang},
\cite{samuelequivalence}, \cite[Exa.~10.3.2]{fulton})
that  if one insists  further in the definition of algebraic equivalence that
the parameter space $T$  be a smooth  projective curve, or   an
abelian
variety, one arrives at the same equivalence relation on
$\operatorname{CH}_i(X)$.   
The first 
goal of the present paper is to secure these results in the case where
the base field is not assumed to be algebraically closed.

\begin{teoalpha}[Proposition \ref{P:algab}]
\label{thmain0}
Let $X$ be a scheme of finite type
 over a 
perfect field $K$, and let
$a \in \operatorname{CH}_i(X)$  be an algebraically trivial cycle class. 
Then 
there exist a smooth  projective integral curve (resp.~abelian variety) $T$ over $K$, a cycle class $ Z\in \operatorname{CH}_{i+\dim T} ( T\cross_K
X)$,   and a pair of
$K$-points $t_0, t_1 \in T(K)$, such that $a =  Z_{t_1} -  Z_{t_0}$ in
$\operatorname{CH}_i(X_K)$. 
\end{teoalpha}

We also consider a notion of algebraic triviality for cycles defined via flat families  (see \S \ref{S:Flat}) and we show that algebraic triviality in this sense is the same as algebraic triviality as defined in \cite[Def.~10.3]{fulton}. 
If  in  Theorem \ref{thmain0} the field $K$ is taken to  be imperfect of characteristic $p$, we obtain the same result taking cycles with $\mathbb Z[1/p]$-coefficients\,; see Proposition \ref{P:algcurveST} and  Theorem \ref{T:main}, as well as Proposition \ref{P:algab} for a related result. 
\medskip

In this paper we are  also interested in a slightly more general question.  Let $L/K$ be a separable algebraic  extension of fields and assume that $a \in
\operatorname{CH}_i(X_L)$ is an algebraically trivial class on
$X_L:=X\times_K L$.   
Theorem \ref{thmain0} shows that  there exist a smooth integral
scheme $T$  of finite type over $L$,   a cycle class  $ Z\in \operatorname{CH}_{i+\dim T} (
T\cross_{L}
X_{L})$, and a pair of
$L$-points $t_0, t_1 \in T(L)$, such that $a =  Z_{t_1} - Z_{t_0}$ in
$\operatorname{CH}_i(X_{L})$, where $T$ may be taken to be a smooth projective curve, or an abelian variety.  
We show that in this situation, one may in fact take $T$ and $Z$ to be defined
over $K$, with $T$ geometrically integral.
  This technical strengthening of Theorem \ref{thmain0}  turns out to be a key step in extending Murre's
work \cite{murre83} on algebraic representatives for varieties over
algebraically closed fields to arbitrary perfect fields\,; see
\cite{ACMVdcg} for details, and specifically Lemma 4.9 and Corollary 4.10 therein.

\begin{teoalpha}[Theorem \ref{T:main}, Proposition \ref{P:algcurveST}]
\label{thmain}
 Let $X$ be a scheme of
finite type over $K$,
let $L/K$ be a separable  algebraic  extension of fields, and let
$a \in \operatorname{CH}_i(X_L)$ be an algebraically trivial cycle 
class. 
Then 
there exist a smooth geometrically integral 
quasi-projective curve
of finite type $T$  over $K$, a cycle class $ Z\in \operatorname{CH}_{i+\dim T} ( T\cross_K
X)$,   and a pair of
$L$-points $t_0, t_1 \in T(L)$, such that $a =  Z_{t_1} -  Z_{t_0}$ in
$\operatorname{CH}_i(X_L)$.  Moreover,  if $K$ is perfect then
 $T$ may be taken to be a smooth geometrically integral projective curve, or an abelian
variety.  
\end{teoalpha}

Clearly Theorem \ref{thmain0} follows from Theorem \ref{thmain}, as the special case where $L=K$.  If in the last assertion of Theorem \ref{thmain} the field $K$ is taken to  be imperfect of characteristic $p$, we obtain the same result using cycles with $\mathbb Z[1/p]$-coefficients\,; see Proposition \ref{P:algcurveST} and Theorem \ref{T:main}. 
Aspects of the proof of the theorem are inspired by the arguments of Weil.  Consequently, in the process of proving the result, we also  provide a modern treatment of some of 
those arguments.      However, there are several places where our treatment diverges significantly from Weil.  In short, using the refined Gysin homomorphism on cycle classes from \cite{fulton}, we are able to develop arguments without the need for a base change of field.
The first place this is apparent is when we allow for taking fibers of cycle classes  over $0$-cycles in the parameter space (see \S \ref{S:algnum}).   This makes it relatively easy to show that one may use curves as the parameter spaces for algebraic triviality, without a base change of fields.  The second place is in showing that one may move from curves as parameter spaces to abelian varieties as parameter spaces.  Here, the refined Gysin homomorphism allows us to use less Brill--Noether theory, and in particular, to work without the base change of field in Weil's arguments.  For comparison, we include a modern treatment of Weil's argument in the appendix, to emphasize where this crucial part of our argument differs from Weil.  
Finally, to prove Theorem \ref{thmain}, there are several new ingredients concerning moving from algebraic triviality over $L$, to questions about cycles over $K$ (see especially Lemma \ref{L:LangK}).  Our results in fact give something stronger than Theorem \ref{thmain}\,;  in particular, we investigate some related notions of algebraic triviality that are motivated by the notion of   $\tau$-algebraic equivalence of \cite{kleiman}.  We direct the reader to Theorem~\ref{T:main} for precise statements.

\subsection*{Acknowledgments}
We thank the referee for helpful suggestions.

\subsection*{Notation and conventions}  We    follow the conventions of \cite{fulton} for the set-up of algebraic cycles on schemes over fields, the formalism of intersection theory, and the definitions of various notions of equivalence of cycles.  
 In particular, following \cite{fulton}, we define a \emph{variety}
 over a field $K$ to be an integral scheme of finite type over
$K$.   
A \emph{curve} over $K$ is a quasi-projective
variety of dimension $1$ over $K$.
The dimension of an irreducible scheme of finite type $T$ over $K$ will be denoted $d_T$.
Given a commutative ring with unit $R$, by  $\operatorname{CH}_i(-)_R$
we mean $\operatorname{CH}_i(-)\otimes_{\mathbb Z} R$. The symbol
$\bar K$ denotes an algebraic closure of the field $K$.

We also fix some notation for closed points of schemes.  
Let $X$ be a scheme of finite type over a field $K$.
 A closed point of $X$ is an element $P$ of the topological space $|X|$ such
that $\{P\}$ is closed.   
 Associated to a closed point $P$ of $X$ is a
$0$-dimensional closed subscheme $\operatorname{Spec}\kappa(P)\to X$, where
$\kappa(P)$ is the residue field at $P$.  We will denote this closed subscheme
of $X$   by using square brackets\,; e.g.,  $[P]$.  We will  denote the $0$-cycle (and cycle class) associated to $P$ also by $[P]$.  
A point  $P\in |X|$ is closed if and
only if the residue field $\kappa(P)$ is a finite extension of $K$.   
If $L$ is a field, and $p:\operatorname{Spec}L\to X$ is an $L$-point of $X$, we
will typically denote the image of $p$ with the corresponding upper case letter\,;
e.g., $P\in |X|$.  The image  point $P$ is
closed  if $L/K$ is algebraic.

\section{Symmetric products of schemes}
\label{S:sym}

This section contains some preliminaries on symmetric products of schemes, which will be used in Section \ref{S:algnum}.  
We expect  the content of this section is well known to the experts, and our primary aim is to fix notation that will be used later.

\medskip

For a scheme $T$, let $\Pi_0(T)$ denote its set of irreducible
components.  For a positive integer $N$, let $S^N(T)$ be the $N$-th
symmetric power of $T$.  

\begin{lem} \label{L:diagsym}
Let $T$ be an integral scheme of finite type over a field $K$, and let $e =
\#\Pi_0(T_{\bar K})$.  For each positive integer $d$, let
\[
S^{\Delta_d}(T_{\bar K}) = \prod_{\bar D \in \Pi_0(T_{\bar K})}
S^d(\bar D) \subset S^{de}(T_{\bar K}).
\]
\begin{enumerate}
\item Then $S^{\Delta_d}(T_{\bar K})$ descends to a geometrically
irreducible component $S^{\Delta_d}(T)$ of $S^{de}(T)$.
\item The natural map $S^{de}(T) \times S^{fe}(T) \to S^{(d+f)e}(T)$
restricts to $S^{\Delta_d}(T) \times S^{\Delta_f}(T) \to
S^{\Delta_{d+f}}(T)$.
\end{enumerate}
\end{lem}

\begin{proof}
  Fix a component $\bar D\in \Pi_0(T_{\bar
K})$, and let $H\subset \gal(K)$ be its stabilizer.  Since $T$ is
irreducible, we have
\[
T_{\bar K} = \bigsqcup_{[\sigma] \in \gal(K)/H} \bar D^\sigma.
\]
Let $e = \#\Pi_0(T_{\bar K})$.  Inside the $de$-th symmetric power
$S^{de}(T)_{\bar K} = S^{de}(T_{\bar K})$ we identify the irreducible
component
\[
S^{\Delta_d}(T_{\bar K}) = \prod_{[\sigma] \in \gal(K)/H} S^d(\bar
D^\sigma).
\]
Since this element of $\Pi_0(S^{de}(T_{\bar K}))$ is fixed by
$\gal(K)$ (and since all geometrically irreducible components
        of a scheme over $K$ are defined over a separable closure of $K$), it descends
 to $K$ as a geometrically irreducible scheme.
 Part (2) is obvious.
\end{proof}

\begin{lem}
\label{L:effectiveKpoint}
Let $T/K$ be a normal integral scheme of finite type over a field $K$, and let
$e = \#\Pi_0(T_{\bar K})$.
An effective zero-cycle of degree $d$ determines a $K$-point of
$S^{\Delta_{d/e}}(T)$.
\end{lem}

\begin{proof} By Lemma \ref{L:diagsym}(2), it is enough to consider the case of closed points on $T$.
It is clear that a closed point $P \to T$ of degree $d$ determines a $K$-point of $S^d(T)$\,; the content of
the assertion is that this point lies in $S^{\Delta_{d/e}}(T)$.  If the
field of definition of the components of $T_{\bar K}$ is Galois over
$K$,
 then the statement is easy.  In general, we need to work a little
harder.

Let $K_c$ be the separable algebraic closure of $K$ inside $K(T)$.  Then $T$ is
a geometrically irreducible scheme of finite type over $K_c$.
Now let $P$ be a closed point of $T$ with residue field $L =
\kappa(P)$, and let $\spec A$ be an open affine subscheme of $T$
which contains $P$\,; let $\idp \subset A$ be the corresponding
ideal.  The natural maps $K\inject A \twoheadrightarrow L$ yield an
inclusion $K_c \inject L$\,; then $K_c$ is contained in $L_s$,
the separable closure of $K$ inside $L$.

The canonical surjection $\pi:A \twoheadrightarrow A/\idp$ yields, after base change, a
surjection $\pi_{\bar K}:A\tensor_K \bar K \twoheadrightarrow A/\idp \tensor_K \bar K$.
On one hand, irreducible components of $\spec (A\tensor_K \bar K)$, and thus of
$T_{\bar K}$, are in bijection with minimal ideals of $A\tensor_K \bar
K$.  
Similarly, the irreducible components of $P_{\bar K} = \spec(A/\idp \tensor_K \bar
K)$ -- that is, the points of $P_{\bar K}$ lying over $P$ -- are in
bijection with minimal ideals of $(A/\idp)\tensor_K \bar K$.  If $\mathfrak{q}$ is a minimal prime of
$(A/\idp)\tensor_K \bar K$, then the component of $T_{\bar K}$
containing the corresponding point is indexed by the minimal prime ideal of
$\pi_{\bar K}\inv(\mathfrak q)$.

On the other hand, irreducible components of $\spec(A\tensor_K \bar K)$
are also indexed by $\hom_K(K_c, \bar K)$,
insofar as 
\begin{equation*}
A\tensor_K \bar K \iso \oplus_{\sigma \in \hom_K(K_c, \bar
K)}A\tensor_{K_c,\sigma}\bar K \,;
\end{equation*}
the irreducible component indexed by $\sigma$ corresponds to the minimal prime ideal containing $\ker(A\otimes_K \bar K \to A \otimes_{K_c,\sigma}\bar K)$.

Similarly, irreducible components of $P_{\bar K}$
are indexed by $\hom_K(L, \bar K)$.  Moreover, we have 
\begin{align*}
(A/\idp)\tensor_K \bar K& \iso \oplus_{\sigma\in \hom_K(K_c, \bar K)}(L\tensor_{K_c,\sigma}
\bar K) \\
&\iso\oplus_{\sigma} \oplus_{\tau \in \hom_{K_c}(L,\bar K): \tau|_{K_c}=\sigma}
L\tensor_{\tau,K_c,\sigma} \bar K.
\end{align*}
Comparison of minimal primes of $A\otimes_K \bar K$ and $(A/\idp)\otimes_K \bar K$ then shows that if $Q \in \Pi_0(P_{\bar K})$
corresponds to $\tau \in \hom_K(L, \bar K)$, then the irreducible
component of $\spec(A \tensor_K \bar K)$ supporting it corresponds to
$\tau|_{K_c}$.

Since, for a given $\sigma \in \hom_K(K_c, \bar K)$,
\[
\# \st{ \tau \in \hom_K(L,\bar K): \tau|_{K_c} = \sigma} = 
\#\st{\tau \in \hom_K(L_s,\bar K):\tau|_{K_c} = \sigma}
= [L_s:K_c]
\]
is independent of $\sigma$, points of $P_{\bar K}$ are uniformly
distributed among the components of $T_{\bar K}$.
\end{proof}

\section{Algebraic equivalence\,: the classical story}\label{S:Classical}

In this section, we recall  various definitions of algebraic equivalence for
cycles classes. These depend on constraints imposed on the smooth parameter
spaces. 
A classical
result of Weil \cite{weil54} says that these definitions all agree when
working over an algebraically closed ground field.
We show in Proposition \ref{P:algab} that this holds over an arbitrary perfect ground field.   We also recall a definition of
flat algebraic equivalence for cycles (as opposed to cycle classes),   show it
is equivalent to  the definition of algebraic equivalence of cycles in 
\cite[Def.~10.3]{fulton},  and recall how it can be used to reformulate rational
triviality.  
   We close this section by showing that, in a suitable
sense, algebraic triviality of cycles on $X$ is independent of the
choice of base field.

\medskip

In this section, we fix a field $K$ and work in the category of schemes of finite
type over $K$.

\subsection{Definition of algebraically trivial  cycle classes}

Let $X$ be a scheme of finite type over some field $K$.
Recall from \cite[\S 1.3]{fulton} that a cycle $\alpha \in \operatorname{Z}_i(X)$ is said to be \emph{rationally equivalent to zero} if there are a finite number of $(i+1)$-dimensional closed integral sub-schemes $W_r$ of $X$, and non-zero rational functions $f_r$ in the function field of $W_r$, such that $\alpha$ is the sum of the cycles associated to $\operatorname{div} (f_r)$. In particular, rational equivalence does not depend on the choice of a base field over which $X$ is of finite type. The collection of $i$-dimensional cycles that are  rationally equivalent to zero form a subgroup of  the group of $i$-dimensional cycles $\operatorname{Z}_i(X)$, and the \emph{Chow group} of $i$-dimensional cycle classes is the quotient of  $\operatorname{Z}_i(X)$ by this subgroup. The cycle class associated to a cycle is the image of that cycle in the Chow group.

\begin{dfn}[{Algebraically trivial cycle class \cite[Def.~10.3]{fulton}}]
\label{D:FultonAT}
Let $X$ be a scheme  of finite type over a field $K$. A dimension-$i$ cycle 
class $a \in \operatorname{CH}_i(X)$ on $X$ is called
\emph{algebraically trivial} if there exist a  smooth integral scheme $T$ of
finite type over $K$,   an $(i+d_T)$-dimensional cycle class $Z\in
\operatorname{CH}_{i+d_T}(T\cross_K
X)$, and $K$-points $t_1,t_0\in T(K)$ such that $Z_{t_1}-Z_{t_0}=a$ in
$\operatorname{CH}_i(X)$.  Here $Z_{t_j}$, $j=0,1$, is the refined Gysin fiber
of the cycle class $Z$ \cite[\S 6.2]{fulton}.  
\end{dfn}

The subset of $ \operatorname{CH}_i(X)$  consisting of algebraically trivial
cycle classes forms a group\,; it thus makes sense to define a dimension-$i$ cycle 
class with $R$-coefficients $a \in \operatorname{CH}_i(X)_R$ to be 
\emph{algebraically trivial} if it is an $R$-linear combination of algebraically trivial cycles in the sense of Definition \ref{D:FultonAT}. For schemes of finite type over $K$, this group
is stable under the usual operations of proper push-forward, flat pull-back,
refined Gysin homomorphisms, and Chern class operations\,; see \cite[Prop.
10.3]{fulton} and Proposition \ref{D:TrivCyc} below.

\begin{rem}
  Several authors require the parameter space $T$ in Definition
  \ref{D:FultonAT} to be separated, or even quasi-projective. For
  instance, Kleiman \cite{kleiman} and Jannsen
  \cite{jannsenequivalence} require the parameter space to be
  quasi-projective (in fact the parameter spaces are taken to be
  projective in \cite{jannsenequivalence}, but over a perfect
    field, by Bertini's theorem and Lemma \ref{L:ProjCurve}, below, 
this is equivalent to
  quasi-projective), and Hartshorne
  \cite{hartshorneequivalence} seems to require the parameter space to
  be separated but works over an algebraically closed field.  Note
  that Fulton \cite{fulton} defines a \emph{variety} as an integral
  scheme of finite type over a field, and thus does not require the
  parameter space to be separated (cf.~\cite[B.2.3]{fulton})\,; we
  have followed this more general convention.  As we will recall in
  Proposition \ref{P:algab}, these conventions all lead to the same
  notion of triviality over a perfect field.
\end{rem}

\begin{rem}\label{R:geomint}
In Definition \ref{D:FultonAT} the parameter scheme $T$ is always 
geometrically integral.  This is because  a smooth integral scheme of finite
type over a field $K$ that admits a $K$-point is geometrically integral.
\end{rem}

\subsection{Related notions of triviality for cycle classes} 

There are other notions of triviality used in the literature that are  obtained
by imposing constraints on the parameter space.  They are all equivalent over an
algebraically closed field by Weil \cite[Lem.~9]{weil54}\,; our main
result, Theorem~\ref{thmain0} (Proposition \ref{P:algab}), extends Weil's result by showing that
they are all equivalent 
over a perfect field.

\begin{dfn}[Parameter space algebraically trivial cycle
classes]\label{D:TrivCyc}
Let $X$ be a scheme  of finite type over a field $K$. 
Let  $a$ be a cycle class in $\operatorname{CH}_i(X)$.
Suppose there exist a   smooth integral scheme $T$ of finite type over $K$,   a
cycle   class $Z\in \operatorname{CH}_{i+d_T}(T\cross_K
X)$, and $K$-points $t_1,t_0\in T(K)$ such that $Z_{t_1}-Z_{t_0}=a$  in
$\operatorname{CH}_i(X)$. 
If $T$ can be taken to be a curve
                (resp.~abelian variety, resp.~projective curve,
                 resp.~quasi-projective scheme,
resp.~separated scheme), then we say that $a$ is \emph{curve
 (resp.~abelian variety, resp.~projective curve, resp.~quasi-projective scheme, resp.~separated scheme) 
trivial.}
\end{dfn}

\begin{pro}[{\cite[Prop.~10.3]{fulton}}] \label{P:comp}
For any of the notions of triviality in Definition  
\ref{D:TrivCyc} above, the  trivial cycles classes  form a subgroup of the group
of all cycle classes.   This subgroup is preserved by the basic operations\,:
\begin{enumerate}
\item proper push-forward\,;
\item flat pull-back\,;
\item refined Gysin homomorphisms\,;
\item Chern class operations.
\end{enumerate}
\end{pro}

\begin{proof}
The proof is the same as \cite[Prop.~10.3]{fulton}.  The classes of abelian varieties, quasi-projective schemes, and projective schemes are each stable under products\,; 
in the case of (projective) curves, we in addition  use Bertini's theorem (Theorem
\ref{T:MumBertini}) to cut the product
of (projective) curves back down to a (projective) curve.      That the  four operations (1)-(4) are
preserved follows from the corresponding parts (a)-(d) of
\cite[Prop.~10.1]{fulton}.
\end{proof}

\begin{rem}\label{R:ExtRlin}
By virtue of the previous proposition, given a commutative ring $R$, we can extend the notions of  triviality in   Definitions \ref{D:FultonAT} and 
\ref{D:TrivCyc} to  Chow groups with $R$-coefficients  by extending $R$-linearly.   
For instance, $a\in \operatorname{CH}^i(X)_R$ is said to be curve
trivial if it is  an $R$-linear combination of curve trivial classes
in $\operatorname{CH}_i(X)$ (as in  Definition \ref{D:TrivCyc}).    All of the results in Section \ref{S:Classical} hold in this more general setting. 
\end{rem}

\begin{rem}\label{R:relations}
We have the following elementary implications among the notions of triviality\,:
$$
\xymatrix@C=1.4em{
\text{ab.~var.~triv.} \ar@{=>}[r]& \text{proj.~curve~triv.} \ar@{=>}[r]& \text{curve triv.}  \ar@{<=>}[r]& \text{q.p.~sch.~triv.} \ar@{=>}[r]
&\text{sep.~sch.~triv.} \ar@{=>}[r]&\text{alg.~triv.}\\
}
$$
Indeed, all the implications are immediate from the definitions, except for the
implications ``Abelian variety trivial $\implies$ Projective curve trivial'' and ``Quasi-projective scheme trivial $\implies$ Curve trivial''. These follow from
applying Bertini's theorem (Theorem \ref{T:MumBertini}). 
\end{rem}

\subsection{Projective curve triviality}
Since every smooth integral curve over a field has a regular
projective model (the normalization of any projective completion),
over a perfect field we have ``Projective curve trivial'' $\iff$
``Curve trivial''.  However, over an imperfect field, there exist
projective curves that are regular but not smooth.   Such a curve admits no surjective
morphism from a smooth projective scheme (this follows from Bertini's
theorem and \cite[Tag 0CCW]{stacks-project})\,; in particular, it
admits no smooth projective model.  Thus the reverse implication over
imperfect fields is less clear.  The following lemma addresses this\,:

\begin{lem}
\label{L:ProjCurve}
Let $X/K$ be a scheme of finite type, and let $a\in \chow_i(X)$ be a curve trivial cycle class.
\begin{enumerate}
\item If $K$ is perfect, then $a$ is projective curve trivial.
\item If $K$ is imperfect of characteristic $p$, then there exists
  some $n$ such that $p^na$ is projective curve trivial.
\end{enumerate}
\end{lem}

\begin{proof}
Suppose $T$ is a smooth curve and $Z$, $t_1$ and $t_0$ are as in
Definition \ref{D:TrivCyc}. Let $\til T$ be the regular projective 
model of $T$, and let $Z'$ be a pre-image 
of the cycle class $Z$ under the surjective \cite[Prop. 1.8]{fulton} restriction map $ \operatorname{CH}_{i+d_T}(\widetilde T\cross_K
X) \to  \operatorname{CH}_{i+d_T}(T\cross_K
X)$.  Then  $Z_{t_1}'- Z_{t_0}'=Z_{t_1}-Z_{t_0}$.

First, suppose $K$ is perfect.  Then $\til T$ is smooth, and we are done.

Second, suppose $K$ is imperfect of characteristic $p$.  Let $\til
T^{(p^n)}$ denote the pullback of $\til T$ by the iterated Frobenius
map $\lambda \mapsto \lambda^{p^n}$ of $K$.  The iterated relative
Frobenius map $F_{\til T/K}^n: \til T \to \til T^{(p^n)}$ is finite and
flat of degree $p^n$, and for sufficiently large $n$ the 
normalization $W_n$ of $\til T^{(p^n)}$ is smooth \cite[Lem.~1.2]{schroer09}.  Fix one such $n$, and let $\nu:W_n \to \til
T^{(p^n)}$ be the normalization map.  Since $\til T$ is smooth 
at each
$K$-point $t_i$, there are unique $K$-points $w_i$ of $W_n$ above each
$F^n_{\til T/K}(t_i)$.  Let $Y = \nu^*F^n_{\til T/K}(Z')$.  Then
$Y_{w_1}-Y_{w_0} = p^n(Z_{w_1}-Z_{w_0})$.
\end{proof}

\subsection{Algebraically trivial cycles  using  $0$-cycles of degree
$0$}\label{S:algnum}

Algebraic equivalence of cycle classes is defined by using the difference of two
Gysin fibers over two $K$-points of the parameter scheme. Working more
generally with $0$-cycles of degree $0$ on the parameter scheme provides more
flexibility. Lemma \ref{L:algnum} and Proposition \ref{P:algcurve} below say
that by doing so we in fact obtain the same equivalence relation on cycles. In
addition, Proposition \ref{P:algcurve}  establishes that algebraic triviality 
coincides with quasi-projective scheme triviality.\medskip

Let $P_1,\ldots, P_r$ be closed points of a scheme $X$ of finite type over a
field $K$, with
residue fields
$\kappa(P_i)$, respectively.
By definition,
the \emph{degree} 
of the zero-cycle $\beta = \sum_i n_i[P_i] \in \operatorname{Z}_0(X)$ is the
integer $\deg \beta := \sum_in_i[\kappa (P_i):K]$.   If $X$ is proper over $K$,
this agrees with the  definition via proper push-forward to
$\operatorname{Spec}K$ \cite[Def.~1.4]{fulton}.  
Given a smooth scheme $T$ of finite type over $K$, a cycle class  $Z \in
\operatorname{CH}_{i+d_T}(T\times_K X)$, and a zero-cycle $\beta = \sum_i n_i[P_i] \in
\operatorname{Z}_0(T)$,   we define 
\begin{equation}\label{E:FibPush}
Z_{\beta} := \sum_i n_i Z_{[P_i]} \in \operatorname{CH}_i(X), 
\end{equation}
where $Z_{[P_i]}$ denotes the image of the refined Gysin fiber of $Z$ over the
closed subscheme  $[P_i]$ of $T$ (which is regularly embedded since $T$ is smooth) pushed forward along the projection
$\operatorname{Spec}\kappa(P_i)\times_T(T\times_KX)=:X_{[P_i]} \to X$\,; this is
also the finite morphism obtained by pulling back $X$ along the finite morphism
$\operatorname{Spec}\kappa(P_i)\to \operatorname{Spec}K$.

\begin{lem}\label{L:algnum}
Let $X$ be a scheme of finite type over a field $K$. Let $T$ be a smooth
integral quasi-projective
     scheme over $K$ and let $\beta\in
\operatorname{Z}_0(T)$
with $\deg \beta =0$. Let $Z\in \operatorname{CH}_{i+d_T}(T\times_KX)$ be a
cycle class on $T\times_K X$. Then the cycle class $Z_\beta \in
\operatorname{CH}_i(X)$ is quasi-projective scheme trivial.
\end{lem}
\begin{proof} 
Applying  Bertini's theorem (Theorem \ref{T:MumBertini}) gives 
a  smooth  integral quasi-projective curve  $g:C\hookrightarrow
T$ over $K$, passing through the support of the zero-cycle $\beta$.
Denote the corresponding zero-cycle of $C$  by $\gamma$.  For the   cycle
$Z':=g^!V\in \operatorname{CH}_{i+1}(C\cross_KX)$, we have   $Z'_\gamma=Z_\beta$  in
$\operatorname{CH}_i(X)$ (\cite[Thm.~6.5]{fulton}).   Thus we may 
and do assume that
$T$ is a smooth integral
     quasi-projective
curve.  

Recall that for any positive integer $N$, we let $S^NT$ denote the  $N$-th symmetric power of
$T$.  
Since $T$ is a smooth integral quasi-projective
curve, $S^NT $ is a smooth  integral quasi-projective
scheme. Let $S^NZ$ be the induced cycle on $S^NT \times_K X$.  
More precisely, view $S^NT$ as a component of the Hilbert scheme of $T$ (see e.g., \cite[Thm.~9.3.7, Rem.~9.3.9]{kleimanPIC}), and let $\mathcal D\subseteq S^NT\times_KT$ be the universal divisor\,;
 $\mathcal D:= \{(z,x) : x\in \operatorname{Supp} z\}$.  Let $p: S^NT \times_K T \to S^NT$ and $q : S^NT\times_K T \to T$ be the natural projections. Then we can define 
$$
S^NZ := (p|_{\mathcal D} \times \operatorname{id}_X)_*(q|_{\mathcal D} \times \operatorname{id}_X)^* Z \in \operatorname{CH}_{i+N}(S^NT\times_K X),
$$
 where  $(q|_{\mathcal D} \times \operatorname{id}_X)^*$ is the flat pull-back ($\mathcal D$ is smooth and integral, so that the dominant morphism $q|_{\mathcal D} : \mathcal D \rightarrow T$ is flat) and $(p|_{\mathcal D} \times \operatorname{id}_X)_*$ is the proper push-forward ($p|_{\mathcal D} :\mathcal D \rightarrow S^NT$ is finite, in particular proper). The fiber of $S^NZ$ over
$r_1+\cdots + r_N \in S^NT$ is then $\sum Z_{r_i}$.

  Note that an effective zero-cycle of degree $N$ on $T$ determines a rational $K$-point of $S^{N}T$.
The key point now is that, even though the smooth integral curve $T$ over $K$  may not have a $K$-point and may be
geometrically reducible, there is an irreducible component of a symmetric
product of $T$ that has a $K$-point and is
                geometrically irreducible.
                
If $T_{\bar K}$ has exactly $e$ irreducible components, then for each
positive integer $d$ we have a geometrically irreducible component
(Lemma \ref{L:diagsym})
\[
S^{\Delta_d}(T) \subset S^{de}(T)
\]
of the symmetric product, parameterizing those $de$-tuples with equal
weight in each geometric component of $T_{\bar K}$. The degree of any zero-cycle on $T$ is a multiple of $e$, and
an effective zero-cycle of degree $de$ on $T$ determines then a rational $K$-point $P \to
S^{\Delta_{d}}T$ (Lemma \ref{L:effectiveKpoint}).

  Let us then write $\beta = \sum_{i=1}^r n_i[Q_i] - \sum_{j=1}^s m_j[R_j]$,
where $n_i>0$ and $m_j>0$ for all $i,j$. Let $N := \sum_{i=1}^r n_i[\kappa
(Q_i):K] = \sum_{j=1}^s m_j[\kappa(R_j):K]$, and consider $t_0 :=  \sum_{i=1}^r
n_iQ_i$ and $t_1 := \sum_{j=1}^s m_jR_j$ \emph{viewed as $K$-points} of  $S^{\Delta_{N/e}}T$.
Then we find that $Z_\beta = (S^{\Delta_{N/e}}Z)_{t_0} - (S^{\Delta_{N/e}}Z)_{t_1}$, where  $S^{\Delta_{N/e}}Z$ denotes the component of $S^{N}Z$ over $S^{\Delta_{N/e}}T\times X$. In other words,
$Z_\beta$ is quasi-projective scheme trivial.
\end{proof}

A corollary is the following further implication   that can be added to the
diagram of Remark \ref{R:relations}\,:

\begin{pro} \label{P:algcurve}
Let $X$ be a scheme of finite type over a field $K$, and let $a\in
\operatorname{CH}_i(X)$ be an algebraically trivial cycle class. Then $a$ is
quasi-projective scheme trivial. Thus, given Remark  \ref{R:relations}, we have
the following implications among the notions of triviality\,:
$$
\xymatrix@C=1.3em{
\text{ab.~var.~triv.} \ar@{=>}[r]& \text{proj.~curve~triv.} \ar@{=>}[r]& \text{curve triv.}  \ar@{<=>}[r]& \text{q.p.~sch.~triv.} \ar@{<=>}[r]
&\text{sep.~sch.~triv.} \ar@{<=>}[r]&\text{alg.~triv.}
}
$$
\end{pro}

\begin{proof}  This is motivated by \cite[Exa.~10.3.2]{fulton}.  
Let $X$ be a scheme of finite type over $K$, and
let $a \in \operatorname{CH}_i(X)$ be an algebraically trivial cycle class. 
From the definition, there exist a smooth integral scheme $T$  of finite type
over $K$,  a pair of
$K$-points $t_1, t_0 \in T(K)$, and a cycle class $Z'\in
\operatorname{CH}_{i+d_T}(T\cross_K
X)$, such that $a = Z'_{t_1} - Z'_{t_0}$.     As mentioned in Remark \ref{R:geomint}, $T$ is
geometrically integral, since it admits $K$-points.

By Bertini's theorem (Corollary \ref{C:MumBertini}), 
there are  smooth geometrically integral quasi-projective curves
$C_1,\ldots,C_r\subseteq T$ over $K$ for some $r$, and closed points
$T_{j_1},T_{j_0}$ in $C_j$ for $j=1,\ldots,r$, such that $ T_{j_1}= T_{(j+1)_0}$
for $j=1,\dots,r-1$, and   $T_{1_0}$ is the closed point associated to $t_0$,
and  $T_{r_1}$ is the closed point associated to  $t_1$.   
Then if we let $Z^{(j)}\in \operatorname{CH}_{i+1}(C_j\times_K X)$ be the refined
Gysin restriction of $Z'$ to $C_j$, $j=1,\dots, r$, we have 
$$
a=(Z'_{ t_1}-Z'_{ t_0})=\sum_{j=1}^r(Z^{(j)})_{[ T_{j_1}]}-(Z^{(j)})_{
[T_{j_0}]} \in \operatorname{CH}_i(X).
$$
Here,  as defined above in \eqref{E:FibPush}, the notation  $(Z^{(j)})_{[
T_{j_\ell}]}$ indicates the push forward under the finite morphism 
$X_{[T_{j_\ell}]}\to X$ of the refined Gysin fiber of $Z^{(j)}$ in
$\operatorname{CH}_i(X_{[T_{j_\ell}]})$.

With this set-up, we now fix\,:
\begin{align}
S& = C_1\times_K \cdots \times_K C_r, \quad \operatorname{pr}_j : S \to C_j \ \
\text{the $j$-th projection},\\
s_\ell& = [T_{1_\ell}]\times_K \cdots \times_K [T_{r_\ell}]\in
\operatorname{CH}_0(C_1\times_K \cdots \times_K C_r), \ \ \ell=0,1, \\
V_j& =\operatorname{pr}_j^*Z^{(j)} \in \operatorname{CH}_{i+r}(S\times _KX), \ \
j=1,\dots ,r, \\
V& =\sum_{j=1}^rV_j \in \operatorname{CH}_{i+r}(S\times _KX).
\end{align}
The scheme $S$ is a smooth  geometrically integral  quasi-projective 
variety over $K$,
and $s_1$ and $s_0$ are closed $0$-dimensional subschemes of $S$\,;  we have that
$V_{ s_1}-V_{ s_0}=Z_{t_1}-Z_{ t_0}=a$. Since $\deg s_0 = \deg s_1$, we can
apply Lemma \ref{L:algnum} to conclude that $a$ is quasi-projective scheme
trivial.
\end{proof}

A zero-cycle $\beta$ is said to be \emph{numerically trivial} if $\deg \beta =
0$.
The following proposition, which is a corollary of Lemma \ref{L:algnum},  is
certainly well-known (e.g., \cite[19.3.5, p.386]{fulton} in the case $K=\cx$)\,:

\begin{pro}
Numerical and algebraic triviality  agree for $0$-cycles on smooth
integral projective schemes over $K$.
\end{pro}
\begin{proof}
If $X$ is a smooth geometrically integral quasi-projective  scheme  over $K$, we apply Lemma \ref{L:algnum} to $T=X$ and $Z = \Delta_X$, the
diagonal sitting in $X\times_K X$, and we find that any zero-cycle of degree~$0$
is algebraically trivial. Conversely, assuming furthermore that $X$ is proper, an algebraically trivial zero-cycle is
numerically trivial by the principle of conservation of number \cite[\S
10.2]{fulton}. 
\end{proof}

\begin{rem}
The proposition is not true without the smoothness hypothesis (see Example
\ref{E:RNC}) or the  hypothesis that $X$ be proper (e.g., $X=\mathbb A^1_K$, where all $0$-cycles are algebraically trivial).
\end{rem}

\begin{exa}\label{E:RNC}
	Consider the rational nodal curve  $X=\{zy^2+zx^2-x^3=0\}\subseteq  \mathbb
	P^2_{\mathbb Q}$ defined over $\mathbb Q$.  Let $N=[0:0:1]$ be  the node, and
	let   $P=[1:0:1]$.  
One can check that $\operatorname{Z}_0(X)/(\text{algebraic equivalence})
	\cong \mathbb Z[P] \oplus (\mathbb Z/2\mathbb Z)([N]-[P])$. In particular $[N]-[P]$ is not algebraically trivial.  
	 The idea is that, for a closed point $O$ of $\mathbb{P}^1_\rat$ with residue field of degree $d$ over~$\rat$, we have $\operatorname{Z}_0(\mathbb P^1_{\mathbb Q}\backslash O)/(\text{algebraic equivalence})= \operatorname{CH}_0(\mathbb{P}^1_\rat \backslash O) = \mathbb Z / d\mathbb Z$ (this is proved using the localization exact sequence for Chow groups). Now, under the normalization morphism $\mathbb{P}^1_{\rat} \to X$, the pre-image of the node $N$ (which is a closed point with residue field $\rat$) is a closed point in $\mathbb{P}^1_\rat$ with residue field $\rat(i)$, and one can yet again use the localization exact sequence to conclude. Alternately, one could apply directly \cite[Exa.~1.8.1, Exa.~10.3.4]{fulton} by considering the node sitting inside $X$.

	 Note in contrast  that over $\mathbb Q(i)$ the same computation shows that the cycle $[N]-[P]$ is algebraically trivial, since  the pre-image of the node in the normalization consists of two  $\mathbb Q(i)$-points. 
As an aside, also note  that $[N]$ and $[P]$ belong to the same component of the Chow scheme of $X/\mathbb Q$. Indeed, one can check that $X$ is semi-normal, so that by \cite[Exe.~I.3.22]{kollar} the Chow scheme $\operatorname{Chow}_{0,1}(X)$ of 0-dimensional subschemes of degree-1 in $X$ coincides with
	$X$.
\end{exa}

\subsection{Abelian variety triviality}

We now complete the proof of Theorem \ref{thmain0} by showing that curve triviality coincides with abelian variety triviality.  

\begin{pro}\label{P:algab}
Let $X$ be a scheme of finite type over a field $K$, and let $a\in
\operatorname{CH}_i(X)$ be a projective curve trivial cycle class.  Then $a$ is abelian variety trivial. 
Thus, given  Proposition \ref{P:algcurve}, we have
the following equivalences among the notions of triviality\,:
$$
\xymatrix@C=1.3em{
\text{ab.~var.~triv.} \ar@{<=>}[r]& \text{proj.~curve~triv.} \ar@{=>}[r]& \text{curve triv.}  \ar@{<=>}[r]& \text{q.p.~sch.~triv.} \ar@{<=>}[r]
&\text{sep.~sch.~triv.} \ar@{<=>}[r]&\text{alg.~triv.}\\
}
$$
In particular, given Lemma \ref{L:ProjCurve}, if $K$ is perfect,
  then all the notions of triviality  above are equivalent. 
\end{pro}

\begin{proof} By assumption  there exist a smooth projective curve $C/K$, 
a cycle class  $Z\in \operatorname{CH}_{i+1}(C\times_KX)$, and $K$-points $p_1,p_0\in C(K)$ such that $Z_{p_1}-Z_{p_0}=a\in \operatorname{CH}_i(X)$.    Since $C$ has a $K$-point,  $C$ is geometrically integral.

 Let $g$ be the genus of $C$.  Take  $N>2g-1$.     Consider as in the proof of Lemma \ref{L:algnum} the
 symmetrized cycle $S^NZ$ on $S^NC\times_KX$\,; its fiber over
 $r_1+\cdots + r_N \in S^NC$ is $\sum Z_{r_i}$.
  Let $q_1$ be the $K$-point of $S^NC$ corresponding to the divisor $Np_1$, and let $q_0$ be the $K$-point of $S^NC$ corresponding to $p_0+(N-1)p_1$.  Then we have that $(S^NZ)_{q_1}-(S^NZ)_{q_0}=Z_{p_1}-Z_{p_0}=a$.

 Since $N>2g-1$, the  Abel--Jacobi  map 
 $$
 AJ_N:S^NC\to \operatorname{Pic}^N_{C/K}
 $$
 is a Zariski locally trivial fibration in projective spaces $\mathbb P^{N-g}_K$\,; i.e., there  is a vector bundle $E$ of rank $N+1-g$ on $\operatorname{Pic}^N_{C/K}$ so that the Abel--Jacobi map coincides with the structure map $ \mathbb PE\to \operatorname{Pic}^N_{C/K}$ (see e.g., \cite[Thm.~8.5(v)]{AK80} and \cite[Exe.~9.4.7]{kleimanPIC}).  Consequently,
the map 
 $$
  AJ_N\times \operatorname{Id}: \ S^NC\times_K X\to \operatorname{Pic}^N_{C/K}\times_K X
 $$
 is the projectivization of the pull-back $E'$ of $E$ to
        $\operatorname{Pic}^N_{C/K}\times _KX$.  Denote  by $\mathcal O(1)$  (resp.~$\mathcal
        O'(1)$) the relatively very ample line bundle on $S^N_{C/K}$ 
       (resp.~$S^N_{C/K}\times _KX$).  Thanks to the projective bundle formula
        \cite[Thm.~3.3(b)]{fulton} we can write
\begin{equation}\label{E:BundForm}
 S^NZ=\sum_{j=0}^{N+1-g}c_1(\mathcal O'(1))^j\cap  (AJ_N\times \operatorname{Id})^*W_j
\end{equation}
 for uniquely determined $W_j\in \operatorname{CH}_{i+1-(N-g)+j}(\operatorname{Pic}^N_{C/K}\times_K X)$. 
 Now, we observe that the terms with~$j$ positive in the above sum do not contribute to Gysin fibers over points of $S^NC$.   More precisely,  if    $q:\operatorname{Spec}K\to S^NC$ is a  $K$-point,  we claim that 
 $$
 q^!S^NZ =q^!(AJ_N\times \operatorname{Id})^*W_0. 
 $$
Indeed,  since  $\mathcal O'(1)$ is obtained from $\mathcal O(1)$ by pull-back, then taking a cycle $D$ representing the cycle class   $c_1(\mathcal O(1))\cap [S^N_{C/K}]$, 
the associated intersections in the sum \eqref{E:BundForm}   can be taken to be  supported on $\operatorname{Supp}(D) \times_K X$, and so the Gysin fibers must be zero.    More precisely,  with the notation in the cartesian  diagram below\,:
 $$
 \xymatrix@R=1.5em{
  X \ar[r]^<>(0.5){q'} \ar[d]_{pr'}& \ar[d]^{pr} X\times_KS^NC\\
  \operatorname{Spec} K \ar[r]^q& S^NC,
 }
 $$
where $pr$ is the second projection, 
 we have 
 \begin{align*}
 q^!(c_1(\mathcal O'(1))^j\cap (AJ_N\times \operatorname{Id})^*W_j)&=c_1(q'^*\mathcal O'(1))^j\cap q^!(AJ_N\times \operatorname{Id})^*W_j & \text{(\cite[Prop.~6.3]{fulton})}\\
 & = c_1(q'^*pr^*\mathcal O(1))^j\cap q^!(AJ_N\times \operatorname{Id})^*W_j\\
 &=c_1(pr'^*q^*\mathcal O(1))^j\cap q^!(AJ_N\times \operatorname{Id})^*W_j\\
 &=0\cap q^!(AJ_N\times \operatorname{Id})^*W_j &
 \end{align*}
where the last equality holds since $q^*\mathcal O(1)=\mathcal O_{\operatorname{Spec}K}$,  
and  $c_1(\mathcal O_{X})=0$.

Having established the claim, then setting  $t=AJ_N\circ q:\operatorname{Spec}K\to \operatorname{Pic}^N_{C/K}$ to be  the composition, we have 
 $
 q^!S^NZ=q^!(AJ_N\times \operatorname{Id})^*W_0 = t^! W_0$  (\cite[Prop.~6.5(b)]{fulton}).
 Therefore, 
  setting $t_j=AJ_N\circ q_j:\operatorname{Spec}K\to \operatorname{Pic}^N_{C/K}$, for $j=1,0$, we have 
 $$
 Z_{p_1}-Z_{p_0}=(S^NZ)_{q_1}-(S^NZ)_{q_0}=(W_0)_{t_1}-(W_0)_{t_0}.
 $$
 Since $C$ has a $K$-point, $\operatorname{Pic}^N_{C/K}$ is an abelian variety, and we are done.
\end{proof}

\begin{rem}
For comparison, in Appendix \ref{S:Weil} we include Weil's argument in
the case where $K=\bar K$ (see Lemma \ref{L:LangKgeom}).    The key
point is that by using the refined Gysin homomorphism, we are able to
get by with less Brill--Noether theory\,; in particular, we do not
need to assume the existence of Brill--Noether-general divisors of a given degree defined over the base field.  
\end{rem}

\subsection{Algebraically trivial cycles and flat  families}\label{S:Flat}  In the definition
of algebraic equivalence for cycle classes (Definition \ref{D:FultonAT}), one
uses refined Gysin homomorphisms. Those are only defined up to rational
equivalence. This leads to the following definition of algebraic equivalence for
cycles\,:

\begin{dfn}[{Algebraically trivial cycle  \cite[Def.~10.3]{fulton}}]
\label{D:FultonATc}
Let $X$ be a scheme  of finite type over a field $K$. A dimension-$i$ cycle  
$\alpha \in \operatorname{Z}_i(X)$ on $X$ is called
\emph{algebraically trivial} if  its associated cycle class $a=[\alpha]\in
\operatorname{CH}_i(X)$ is algebraically trivial.  
\end{dfn}

We can also define a notion of algebraic equivalence for cycles using  flat
families.
Given schemes of finite type  $X$ and $T$ over $K$, integers $m_j$, and 
integral $(i+d_T)$-dimensional subschemes $Z_j$ of $T\times_K X$, a cycle $Z := \sum_j
m_j Z_j \in \operatorname{Z}_{i+d_T}(T\times_K X)$ is said to be \emph{flat} over $T$
if each natural projection morphism $Z_i \to T$ is flat.  If $T$ is smooth over
$K$, and $t\in T(K)$, we define the \emph{flat fiber} $Z_t$  of $Z$ over~$t$ to
be $\sum_j m_j (Z_j)_t$, where $(Z_j)_t$ is the scheme theoretic fiber over~$t$.
Note that in this situation,  the  flat  fiber and the refined Gysin fiber
agree.   Indeed, in the case where $T$ is a smooth curve, this is asserted on the bottom of  \cite[p.176]{fulton}.  More generally, one uses the formula for the Gysin fiber in terms of a Segre class given in  \cite[p.176]{fulton} (and \cite[Prop.~6.1(a)]{fulton}), together with the flatness hypothesis and \cite[Exa.~10.1.2, Exa.~A.5.5]{fulton}, so that one can use the fact that for a regular embedding, the Segre class can be given in terms of the Chern class of a normal bundle \cite[p.74]{fulton}.

\begin{dfn}[Flatly curve trivial cycle] \label{D:cycleAT}
Let $X$ be a scheme  of finite type over a field $K$. 
A dimension-$i$ cycle $\alpha \in \operatorname{Z}_i(X)$ on $X$ is called
\emph{flatly curve trivial} if there exist a smooth integral 
 curve 
$T$   over $K$,   an effective  cycle $Z\in \operatorname{Z}_{i+d_T}( T\times_K X)$
flat over $T$ (i.e., there are integral subschemes $V_j\subseteq T\times_K X$
flat over $T$, $1\le j\le r$,  for some $r$, and  
$Z=\sum_{i=1}^r[V_i]$),
and $K$-points $t_1,t_0\in T(K)$ such that $Z_{t_1}-Z_{t_0}=\alpha$ in
$\operatorname{Z}_i(X)$.  
\end{dfn}

It is clear that if a cycle is flatly curve trivial, then it is 
algebraically trivial.  The following lemma  shows that the converse is true. 
It also shows that one   obtains  the same  notion of triviality for  cycles by
allowing $T$ in Definition \ref{D:cycleAT} to be an arbitrary smooth integral
scheme of finite type over $K$.

\begin{lem}\label{L:CycleEquiv}
A cycle is algebraically trivial if and only if it is flatly curve  trivial.  
\end{lem}

\begin{proof}  
The proof here extends \cite[Exa.~10.3.2]{fulton} to the case where $K$ is not algebraically closed.  
Assume the cycle class $a$ is algebraically  trivial.  By
Proposition \ref{P:algcurve}, we may assume that    $a$ is curve trivial.  
By definition of \emph{curve trivial}, there exist a smooth quasi-projective
 integral  curve 
$C'/K$, a cycle class $Z'\in \operatorname{CH}_{i+1}(C'\cross_KX)$, and $K$-points
$p'_1, p'_0\in C'(K)$ such that  $Z'_{p'_1}-Z'_{p'_0}=a$  in
$\operatorname{CH}_i(X)$.

Let $\widetilde Z\in \operatorname{Z}_{i+1}(C'\times _KX)$ be a 
representative of $Z'$.  As in \cite[p.176]{fulton},  we may simply 
discard all components  of $\widetilde Z$ that do not map 
dominantly onto $C'$, and obtain a cycle $\widetilde Z $ all of 
whose components are flat over $C'$, and with $[\widetilde Z]_{p'_1}-
[\widetilde Z]_{p'_0}=[\alpha]$ in $
\operatorname{CH}_i(X)$.  The issue is that $\widetilde Z$ need 
not be nonnegative (there can be components of $\widetilde Z $ 
taken with negative coefficients).  To rectify this, we argue as
follows.  

With $\widetilde Z$ as above, let  $\widetilde Z=\widetilde Z^+-\widetilde
Z^-$, with $\widetilde Z^+$ and $\widetilde Z^-$ effective,   be the
decomposition of $\widetilde Z$ into positive and negative parts.     Set 
\begin{align*}
S& = C'\times_K  C',\\
s_1& = (p_{1}',  p_{0}') \in S(K), \\
s_0& = (p_0',p'_1)\in S(K)\\
V^+& =\operatorname{pr}_2^*\widetilde Z^+\in \operatorname{CH}_{i+2}(S\times _KX)\\
V^-&=\operatorname{pr}_1^*\widetilde Z^-\in \operatorname{CH}_{i+2}(S\times _KX)\\ 
V& =V^++V^-\in \operatorname{CH}_{i+2}(S\times _KX).
\end{align*}
Since $C'$ is quasi-projective,
 the same is true of $S$.   
Since $C'/K$ is geometrically integral (see Remark~\ref{R:geomint}), we have
that $S$ is geometrically integral.   Since  $C'/K$ is smooth, 
$C'\times_KC'$ is smooth.
Finally, we have  $V_{s_1}-V_{s_0}=Z_{p_1'}-Z_{p_0'}\in
\operatorname{CH}_i(X)$.

To reduce to the case of a curve, we use 
Bertini's theorem  (Theorem
\ref{T:MumBertini}).  We obtain 
a  smooth geometrically integral quasi-projective curve  $g:C\hookrightarrow S$ over $K$,
passing through the $K$-points $s_1$,~$s_0$.
We call these $K$-points $p_1$, $p_0$, respectively on $C$. 
The cycle class $Z=g^!V$ has the property that 
$Z_{p_1}-Z_{p_0}=[\alpha]$  in $\operatorname{CH}_i(X)$.

Before continuing, note that the support of $V$ is obtained from the support of
$\widetilde Z$ by pull-back, so the support of $V$ is flat over $S$.  
Consequently, we may take scheme theoretic fibers of $V$ to obtain the Gysin
fibers, and so we may assume that $Z$ is an effective cycle
(not just a cycle class)  with support flat over $C$.  
The only issue now is that we have  $[Z_{p_1}]-[Z_{p_0}]=[\alpha]$  in
$\operatorname{CH}_i(X)$, not as an equality of cycles.

Nevertheless, since the cycles $Z_{p_1}-Z_{p_0}$ and $\alpha$ are rationally
equivalent,  using \cite[Exa.~1.6.2]{fulton}, 
there is an effective  cycle $W\in \operatorname{Z}_{i+1}(\mathbb P^1_K\times_K X)$
with all components  flat over $\mathbb P^1_K$, a cycle $\beta\in
\operatorname{Z}_i(X)$,  and  
a pair of
$K$-points $q_0, q_1 \in \mathbb P^1_K(K)$  such that 
$$ 
W_{q_1}=Z_{p_1} - Z_{p_0}+\beta, \ \ \ \ W_{q_0}=\alpha+\beta.
$$
Now we take the product $C\times_K\mathbb P^1_K$, pull-back our flat cycles $Z$
and $W$, and use Bertini.  In the same way as before, we obtain a smooth
geometrically integral curve, and a family of effective cycles with flat support
over the curve,  such that the fiber of the family  at one point is 
$$
[Z_{p_0}]+([Z_{p_1}] - [Z_{p_0}]+\beta)
$$
and the fiber at the other point is
$$
[Z_{p_1}]+\alpha+\beta.
$$
The difference of the fibers is then
$$
([Z_{p_1}]+\alpha+\beta)-([Z_{p_0}]+([Z_{p_1}] - [Z_{p_0}]+\beta))=\alpha,
$$
completing the proof.
\end{proof}

\subsection{Rational triviality}
For completeness, in this section we discuss the connection with
rational triviality.   The main point is that by requiring in
Definition \ref{D:TrivCyc} that the parameter space $T$ be a rational
curve, one simply obtains the trivial equivalence relation (see Remark \ref{R:RatTriv}).  
Moreover, while  Definitions \ref{D:FultonAT} and \ref{D:FultonATc} use the definition of rational triviality to define algebraic triviality, we review here that  one can  also use the notion of flatly curve trivial cycles (which does not depend on the definition of rational triviality) to define rational triviality.

\begin{dfn}\label{D:RCT}
Let $X$ be a scheme  of finite type over a field $K$. 
We say    a cycle class  $a\in \operatorname{CH}_i(X)$
is \emph{rational curve trivial} if in the terminology of Definition
\ref{D:TrivCyc}, the parameter space $T$ can be taken to be isomorphic
to (an open subset of) $\mathbb P^1_K$.   A cycle $\alpha \in \operatorname{Z}_i(X)$ is \emph{rational curve trivial} if its cycle class in $\operatorname{CH}_i(X)$ is.
\end{dfn}

\begin{rem} [{\cite[Prop.~1.6]{fulton}}] \label{R:RatTriv} A cycle $\alpha \in \operatorname{Z}_i(X)$   is rational curve trivial if and only if  it is rationally trivial.
Equivalently, 
a cycle class $a \in \operatorname{CH}_i(X)$   is rational curve trivial if and only if $a=0$.
\end{rem}

   Of course, rational curve triviality  still does not provide
an alternative definition of rational triviality for cycles (since Definition \ref{D:RCT} uses cycle classes). However, following \cite[Prop.~1.6]{fulton} it is possible to do this using the notion of flatly curve trivial cycles.

\begin{dfn}\label{D:FRCT}
Let $X$ be a scheme  of finite type over a field $K$. 
We say    a cycle   $\alpha \in \operatorname{Z}_i(X)$
is \emph{flatly rational curve trivial} if in the terminology of
Definition \ref{D:cycleAT}, the parameter curve $T$ can be taken to be
isomorphic to (an open subset of) $\mathbb P^1_K$.  
\end{dfn}

\begin{rem}[{\cite[Prop.~1.6]{fulton}}]
  A  cycle 
 $\alpha  \in \operatorname{Z}_i(X)$ on $X$ is flatly rationally curve trivial    if and only if  it is rationally trivial.
\end{rem}

Finally we have the following observation\,:

\begin{rem}\label{R:ANclass}
Suppose in  Lemma \ref{L:algnum} that  there is an open subset of a  smooth rational curve passing through the support of the zero-cycle $\beta$.
 Then $Z_{\beta}$ is \emph{rational} curve trivial\,; i.e., $Z_\beta =0\in \operatorname{CH}_i(X)$.  Indeed, in the proof of the lemma, one reduces to the case $T=C\subseteq \mathbb P^1_K$ is an open subset, in which case the symmetric product $S^NT$  is an open subset of $\mathbb P^N_K$.  Consequently,  the two $K$-points of $S^NT$ in the proof can be joined by an open subset of $\mathbb P^1_K$, and one concludes using Remark \ref{R:RatTriv}.
\end{rem}

\subsection{Algebraic triviality and base field}\label{S:BaseField}
Let $X$ be a scheme of finite type over a field $K$.  If $K/F$ is a finite extension of fields, then $X$ is also a scheme of finite type over $F$.  Given a cycle class $a\in \operatorname{CH}_i(X)$, one may ask whether it is algebraically trivial viewing $X$ as a scheme over $K$, or as a scheme over $F$.

\begin{pro}\label{P:ind} 
  Let $K/F$ be a finite   extension of fields, let $X$ be a scheme of finite
  type over $K$, and let $a \in \operatorname{CH}_i(X)$ be a cycle
  class on $X$. If $a$ is algebraically trivial on $X$, considered as a scheme of finite
  type over $F$, then it is algebraically trivial on $X$,
  considered as a scheme of finite type over $K$. Moreover, the
  converse holds, if one assumes that $K/F$ is a separable extension.
\end{pro}
\begin{proof}
Assume $a$ is algebraically trivial on $X$, considered as a scheme of finite type over $F$. By Proposition \ref{P:algcurve}, there are a smooth integral quasi-projective curve $C$ over $F$, $F$-points $t_0$ and $t_1$ in $C(F)$, and a cycle $Z$ in $C\times_F X$ such that $a = Z_{t_1} - Z_{t_0}$. Note that $C\times_F X = C_K \times_K X$, and note that since $C$ has an $F$-point, $C$ is geometrically integral, and so $C_K$ is (geometrically) integral. Pulling back the $F$-points $t_0$ and $t_1$ to $K$-points $p_0$ and $p_1$ of $C_K$, and viewing $Z$ as a cycle in $C_K\times_K X$,  one has  that $Z_{t_i}=Z_{p_i}$
so  that $a = Z_{p_1} - Z_{p_0}$, 
i.e., that $a$ is algebraically trivial on $X$, considered as a scheme of finite type over $K$. 

Conversely, assume $K/F$ is separable, and assume $a$ is algebraically trivial on $X$, considered as a scheme of finite type over $K$. By Proposition \ref{P:algcurve}, there are a smooth integral quasi-projective curve $C$ over $K$, $K$-points $t_0$ and $t_1$ in $C(K)$, and a cycle $Z$ in $C\times_K X$ such that $a = Z_{t_1} - Z_{t_0}$. Since $K/F$ is a finite extension, $C$ is a curve over $F$ as well. Moreover, since $K/F$ is separable, $C$ is smooth over $F$. Viewing $C$ and $X$ as schemes over $F$ and $Z$ as a cycle in $C\times_F X$ via the natural inclusion $C\times_KX\hookrightarrow C\times_FX$,
 one has that $a = Z_\beta$, where $\beta = [t_1]-[t_0] \in \operatorname{CH}_0(C)$. By Lemma~\ref{L:algnum}, $a$ is algebraically trivial on $X$, considered as a scheme of finite type over $F$.
\end{proof}

\section{Algebraic equivalence\,: descending the parameter spaces}\label{S:mainresult}

For applications such as \cite{ACMVdcg}, 
our main result concerns a related notion of triviality for cycles after
algebraic base change of field.  More precisely, given a separable algebraic extension of fields $L/K$, a scheme $X$ of
finite type over $K$,  and an algebraically trivial cycle class $\alpha
\in \operatorname{CH}_i(X_{L})$, we show that triviality of $\alpha$ can be
exhibited by a family of cycles defined over $K$ (although the fibers are 
taken over $L$-points).

\subsection{$\rho$-, $\sigma$- and $\tau$-algebraic triviality}

We now introduce three additional notions of triviality.

\begin{dfn}[$\rho$-, $\sigma$- and $\tau$-algebraic triviality] \label{D:sigmatau}
Let $X$ be a scheme  of finite type over a field $K$, and let $L$ be any field
extension of $K$.
A cycle class $a \in \operatorname{CH}_i(X_L)$ is said to be 
\begin{itemize}
\item \emph{$\rho_{L/K}$-algebraically trivial} if there exist a smooth
geometrically integral scheme $T$ 
of finite type over $K$, a cycle class $Z \in
\operatorname{CH}_{i+d_T}(T\times_K X)$, and \emph{points $t_0, t_1 \in T({L})$ such
that $a = Z_{t_1}-Z_{t_0}$ in $\operatorname{CH}_i(X_L)$}.

\item \emph{$\sigma_{L/K}$-algebraically trivial} if there exist a smooth
geometrically integral scheme $T$ of finite type over $K$, a cycle class $Z \in
\operatorname{CH}_{i+d_T}(T\times_K X)$, and \emph{points $t_0, t_1 \in T(\bar{L})$ such
 that $a_{\bar L} = Z_{t_1}-Z_{t_0}$ in $\operatorname{CH}_i(X_{\bar L})$}.

\item \emph{$\tau_{L/K}$-algebraically trivial} if  $Na$ is $\rho_{L/K}$-algebraically trivial for some non-zero integer $N$\,; i.e., there exist a smooth 
geometrically integral scheme $T$ of finite type over $K$, a cycle class $Z \in
\operatorname{CH}_{i+d_T}(T\times_K X)$,  and \emph{points $t_0, t_1
\in T({L})$ such that $Na = Z_{t_1}-Z_{t_0}$ in $\operatorname{CH}_i(X_{L})$}.
\end{itemize}
\end{dfn}

\begin{rem}
In the above definition, the morphisms $t_i:\operatorname{Spec}L\to T$
need not be closed embeddings, or even morphisms of finite type over $K$, and so the refined Gysin fiber $Z_{t_i}$ of \cite[\S
6.2]{fulton} is not immediately well-defined.  In the context of the
conventions in \cite{fulton}, what  we mean is the following.  For a
finite extension $L/K$, the morphism $t_i:\operatorname{Spec}L\to T$ is an lci morphism, in the strong sense of \cite[B.7.6]{fulton}\,; then we can define the refined Gysin fiber as in \cite[\S 6.6]{fulton}. 
 For an arbitrary extension $L/K$, the field $L$ can be written as a direct limit $L=\varinjlim R$ of   $K$-algebras of finite type.   For each $R$, using flat pull-back  one obtains a cycle class $Z_{T_R}\in \operatorname{CH}^i((T\times_KX)_R)$ and one obtains a cycle class $Z_{T_L}\in  \operatorname{CH}^i((T\times_KX)_L)$ using the identification  $\operatorname{CH}^i((T\times_KX)_L)=\varprojlim \operatorname{CH}^i((T\times_K X)_R))$.  Then one can take $Z_{t_i}$ to be the refined Gysin fiber of $Z_L$ over the closed regular embedding $\operatorname{Spec}L\to T_L$ induced by $t_i:\operatorname{Spec}L\to T$.  
\end{rem}

\begin{rem}\label{R:firstcomp}
It is easy to see that a $\rho_{L/K}$-algebraically trivial cycle class is algebraically trivial\,; moreover, if $L=K$, then algebraic triviality is by definition the same as $\rho_{K/K}$-algebraic triviality. It is clear that  a $\rho_{L/K}$-algebraically trivial cycle class is $\sigma_{L/K}$-algebraically trivial.
Over an algebraically closed field $K=\bar K$, algebraic triviality is by definition the same as $\sigma_{K/K}$-algebraic triviality.  Our main result, Theorem \ref{T:main}, compares all these notions.    
\end{rem}

\begin{rem} \label{R:RNCsigma}
Over an arbitrary field $K$, an algebraically trivial cycle class $a\in \operatorname{CH}_i(X)$ is by definition both
$\sigma_{K/K}$- and $\tau_{K/K}$-algebraically trivial.
The converse fails\,: the rational nodal curve over $\rat$ of Example \ref{E:RNC} with the $0$-cycle $[N]-[P]$ provides a counter-example.  Indeed,  $2([N]-[P])$ is algebraically trivial, so that $[N]-[P]$ is $\tau_{K/K}$-algebraically trivial, and $[N]-[P]$ becomes rationally  trivial after base-changing to an algebraic closure of $\rat$,  so that  $[N]-[P]$ is also $\sigma_{K/K}$-algebraically trivial.  However,  $[N]-[P]$ is not algebraically trivial.
\end{rem}

\begin{rem}
In \cite{kleiman}, one finds the notion of $\tau$-algebraically  trivial cycles (see also \cite[p.374]{fulton})\,; in
our notation, this is $\tau_{K/K}$-algebraic triviality. 
\end{rem}

\begin{dfn}\label{D:TrivCycST}
Similarly to Definition \ref{D:TrivCyc}, one may define refined notions of
$\rho$-, $\sigma$- and $\tau$-algebraic triviality, \emph{e.g.},
\emph{$\sigma_{L/K}$-abelian variety trivial}, \emph{$\tau_{L/K}$-curve 
trivial}, \emph{etc}.  
\end{dfn}

In the next subsections we will establish further relationships among the notions of triviality.

\subsection{First properties of $\rho$-, $\sigma$- and $\tau$-triviality}

We start by extending the classical results in~\S \ref{S:Classical} to the new notions of triviality.

\begin{pro} \label{P:basic}
For any of the notions of triviality in Definition \ref{D:sigmatau} and Definition \ref{D:TrivCycST} above, the 
trivial cycles classes  form a subgroup of the group of all cycle classes. 
  In
the category of schemes of finite type over $K$, this subgroup is preserved by
the basic operations\,:
\begin{enumerate}
\item proper push-forward\,;
\item flat pull-back\,;
\item refined Gysin homomorphisms\,;
\item Chern class operations.
\end{enumerate}
\end{pro}

\begin{proof}
The proof is the same as \cite[Prop.~10.3]{fulton}.  
In the case of curves, we in addition  use Bertini's theorem to cut the product
of curves back down to a  curve.      That the  four operations (1)-(4) are
preserved follows from the corresponding parts (a)-(d) of
\cite[Prop.~10.1]{fulton}.  
\end{proof}

\begin{rem}\label{R:sigmatauR}
By virtue of the previous proposition, given a commutative ring $R$, we extend the notions of  triviality in   Definitions   \ref{D:sigmatau} and  \ref{D:TrivCycST} 
to Chow groups with $R$-coefficients  by extending $R$-linearly.  
\end{rem}

\begin{rem}\label{R:tauQQ}
 A cycle $a\in \operatorname{CH}_i(X_L)$ is $\tau_{L/K}$-algebraically trivial if and only if $a$, viewed as an element of  $\operatorname{CH}_i(X_L)_{\mathbb Q}$, is $\rho_{L/K}$-algebraically trivial.  
\end{rem}

We also have the following analogue of Theorem \ref{thmain0}.

\begin{pro} \label{P:algcurveST} Let $X$ be a scheme of finite type
  over a field $K$, let $L/K$ be an algebraic extension of fields, and
  let  $R$ be a commutative ring. (If $K$ is imperfect of
    characteristic $p$, assume that $p$ is invertible in $R$.)
If $a\in \operatorname{CH}_i(X_L)_R$
  is a $\rho_{L/K}$- (resp.$\sigma_{L/K}$-,
  resp.$\tau_{L/K}$-)algebraically trivial cycle class, then $a$ is
  $\rho_{L/K}$- (resp.$\sigma_{L/K}$-, resp.$\tau_{L/K}$-)abelian
  variety trivial. Moreover, we have the following implications among
  the notions of triviality (with $R$-coefficients)\,:
\begin{equation*}
\begin{multlined}
\xymatrix@R=.5em@C=1.8em{
\rho_{L/K}\text{-ab.~var.} \ar@{<=>}[r]   & \rho_{L/K}\text{-curve}  \ar@{<=>}[r]  & \rho_{L/K}\text{-q.p.~sch.} \ar@{<=>}[r] 
 &\rho_{L/K}\text{-sep.~sch.} \ar@{<=>}[r]   &\rho_{L/K}\text{-alg.}   \\
\sigma_{L/K}\text{-ab.~var.} \ar@{<=>}[r]   & \sigma_{L/K}\text{-curve}  \ar@{<=>}[r]  & \sigma_{L/K}\text{-q.p.~sch.} \ar@{<=>}[r] 
&\sigma_{L/K}\text{-sep.~sch.} \ar@{<=>}[r]   &\sigma_{L/K}\text{-alg.}   \\
\tau_{L/K}\text{-ab.~var.} \ar@{<=>}[r]& \tau_{L/K}\text{-curve}  \ar@{<=>}[r]& \tau_{L/K}\text{-q.p.~sch.} \ar@{<=>}[r]
&\tau_{L/K}\text{-sep.~sch.} \ar@{<=>}[r]&\tau_{L/K}\text{-alg.}\\
}
\end{multlined}
\end{equation*}
\end{pro}

\begin{proof} 
 As in Remark \ref{R:relations}, using Bertini's Theorem, we see that the implications from left to right hold. 
The proof of Proposition \ref{P:algcurve} can be easily adapted to show that $\rho_{L/K}$-(resp.~$\sigma_{L/K}$-, resp.~$\tau_{L/K}$-)triviality implies $\rho_{L/K}$-(resp.~$\sigma_{L/K}$-, resp.~$\tau_{L/K}$-)curve triviality.  
The proof of Proposition \ref{P:algab} can also be easily adapted to
show  (using Lemma \ref{L:ProjCurve} if $K$ is imperfect)
that $\rho_{L/K}$-(resp.~$\sigma_{L/K}$-, resp.~$\tau_{L/K}$-)curve
triviality implies $\rho_{L/K}$-(resp.~$\sigma_{L/K}$-,
resp.~$\tau_{L/K}$-)abelian variety triviality  with $R$-coefficients.
\end{proof}

\subsection{Main theorem} 
Our main result is to compare algebraic triviality with $\rho$-, $\sigma$- and
$\tau$-algebraic triviality\,:

\begin{teo}\label{T:main}
Let $X$ be a scheme  of finite type over a field $K$, let $L$ be an algebraic 
extension of $K$, and let $R$ be a commutative ring.
Consider a cycle class $a \in \operatorname{CH}_i(X_L)_R$, and consider the following statements\,: 
\begin{enumerate}[(i)]
\item  The cycle $a$ is
algebraically trivial\,; \label{T:main(i)}
\item  The cycle $a$ is $\rho_{L/K}$-algebraically trivial\,; \label{T:main(i')}
\item  The cycle $a$ is $\sigma_{L/K}$-algebraically trivial\,; \label{T:main(ii)}
\item   The cycle $a$ is $\tau_{L/K}$-algebraically trivial. \label{T:main(iii)}
\end{enumerate}
Then  (\ref{T:main(i')}) $\implies$ (\ref{T:main(i)})  and  (\ref{T:main(i')}) $\implies$  (\ref{T:main(ii)}) $\implies$ (\ref{T:main(iii)}). If in addition one assumes that either $L/K$ is separable or $\operatorname{char}(K)$ is invertible in $R$, then (\ref{T:main(i)}) $\implies$ (\ref{T:main(i')}).
\end{teo}

Note from Remark \ref{R:firstcomp} that the only implications left to show in Theorem \ref{T:main} are the implication   $\eqref{T:main(ii)}\implies \eqref{T:main(iii)}$, and the implication $\eqref{T:main(i)} \implies \eqref{T:main(i')}$. These are proved respectively in Lemma \ref{L:taucurve}, and Lemma \ref{L:LangK}.
Note that Theorem \ref{thmain} is the implication  $\eqref{T:main(i)} \implies \eqref{T:main(i')}$ of Theorem \ref{T:main} in the case where $L/K$ is separable.
 Note also that complex Enriques surfaces provide  counter-examples to the implication \eqref{T:main(iii)} $ \implies$ \eqref{T:main(ii)} (see Example \ref{E:enriques} below), and that the rational nodal curve over $\rat$ of Example \ref{E:RNC}  provides a counter-example to the implication   \eqref{T:main(ii)} $ \implies$ \eqref{T:main(i')}  (see Remark \ref{R:RNCsigma}).
 
 \begin{exa}\label{E:enriques}
 In general, $\tau_{L/K}$-triviality does not imply $\sigma_{L/K}$-triviality.
Consider for instance a non-zero torsion element in the
N\'eron--Severi group of a smooth complex projective  variety (e.g., $c_1(S) \in
\operatorname{CH}_1(S)$ for $S$ a complex Enriques surface)\,; such a cycle is
$\tau_{\cx/\cx}$-trivial, but it is certainly not $\sigma_{\cx/\cx}$-trivial.
\end{exa}

\begin{cor}\label{C:STDiag} 
 Let $X$ be a scheme  of finite type over a field $K$, let $L$ be a separable algebraic 
 extension of $K$, and let $R$ be a commutative ring.
 For  cycle classes in $\operatorname{CH}_i(X_L)_R$\,:
 \begin{equation}\label{E:SigTauDiag}
 \begin{multlined}
 \xymatrix@R=1.8em@C=1.8em{
  \text{alg. triv.} \ar@{<=>}[r]& \rho_{L/K}\text{-alg. triv.}  \ar@{=>}[r]& \sigma_{L/K}\text{-alg. triv.} \ar@{=>}[r]
  &\tau_{L/K}\text{-alg.triv.} \\
 }
 \end{multlined}
 \end{equation}
 Moreover, if  $K=\bar K$ then the middle implication is an equivalence, and if  $\mathbb Q\subseteq R$ then all 
 implications are  equivalences.   \qed
\end{cor}

\subsubsection{Proof of $\eqref{T:main(i)} \implies \eqref{T:main(i')}$ in Theorem \ref{T:main}} \label{S:i=>ii}
\begin{lem}\label{L:LangK}
Let $X$ be a scheme of finite type over a field $K$, 
let $L/K$ be an algebraic   field extension, let $R$ be a commutative ring, and let $a\in
\operatorname{CH}_i(X_ L)_R$ be an algebraically trivial cycle class with $R$-coefficients. Assume either that $L/K$ is separable, or that $\operatorname{char}(K)$ is invertible in $R$.

Then there  are a smooth geometrically integral  curve $C/K$,
admitting a $K$-point,  a cycle class  $Z\in
\operatorname{CH}_{i+1}(C\cross_K X)$,
   and $L$-points  $p_1, p_0$ on $C$ such that $Z_{p_1}-Z_{p_0}=a_{L}$ in $\operatorname{CH}_i(X_{L})_R$.
    In other words, $a$ is $\rho_{L/K}$-curve trivial.
\end{lem}

\begin{proof}
It is easy to see that there is a finite field  extension $F/K$ and an
algebraically trivial  cycle $a_F\in \operatorname{CH}_i(X_F)_R$ such that
$a=(a_F)_L$.  Thus we may assume that
$L/K$ is finite. In the case where $K$ has positive characteristic, say $p$, denote $L^{sep}$ the largest separable extension of $K$ contained in $L$ and let $p^r$ be the degree of the purely inseparable extension $L/L^{sep}$. Then, for the natural projection proper morphism $f : X_L \to X_{L^{sep}}$, we have $p^r a = f^*f_* a$.  Therefore, if $p$ is invertible in~$R$, it is enough to show that $f_*a$ is $\sigma_{L^{sep}/K}$-curve trivial.
Thus it suffices to prove the lemma for the special case where $R = \mathbb Z$ and
$L/K$ is a finite separable extension.

The proof is similar in spirit to that of Proposition \ref{P:ind} (together with Lemma \ref{L:algnum}), but slightly different in the sense that the Gysin fibers are taken over $L$-points rather than $K$-points.
We proceed in two steps.

(1) By Proposition \ref{P:algcurve},    there exist a smooth  integral quasi-projective  
curve   $D/L$, 
a cycle class  $Y\in \operatorname{CH}_{i+1}(D\cross_LX_L)$,  and $L$-points $q_1,
q_0$  on $D$ such that  $ Y_{q_1}-Y_{q_0}=a_{L}$.
Now since $L/K$ is finite and separable, $D$ is a smooth quasi-projective 
curve over $K$ as well. 
Now we make
the trivial observation that
$$
D\times_LX_L=D\times_L (L\times_K X)=D\times_KX.
$$
Therefore, $Y\in \operatorname{CH}_{i+1}(D\times_KX)$, and, taking fibers over $L$-points,  we have $
Y_{q_1}-Y_{q_0}=a_{L}$.  

(2) Viewed as a smooth curve over $K$, $D$ is integral but, if
$K\subsetneq L$, 
$D$ is not geometrically irreducible, and $D(K)$ is empty.
Nonetheless, as in the proof of Lemma \ref{L:algnum}, we can identify
a geometrically irreducible curve (over $K$) inside a suitable
symmetric power of $D$.  Let $e := [L:K]$\,; then $D_L = D\times_KL$
has $e$ irreducible components $\bar D^{(1)}, \cdots, \bar D^{(e)}$,
each of which is geometrically irreducible.  Inside the $e$-fold
symmetric product $S^e(D)$ of $D$ (over $K$), we have the
geometrically irreducible subvariety $S^{\Delta_1}(D)$ (Lemma \ref{L:diagsym})\,; and each $q_i \in D(L)$ gives rise to a point
$\st{x_i^1, \cdots, x_i^e} \in S^{\Delta_1}(D)(K)$, where $x_i^j \in
\bar D^{(j)}$. 

Now let  $S^eY$ be the symmetrized induced cycle on $S^e(D)\times_K X$ and let $Z$ be the restriction of  $S^eY$ to $S^{\Delta_1}(D)\times X$. Let $p_1 = \st{x_1^1,x_1^2, \cdots, x_1^e} \in S^{\Delta_1}(D)(K)\subset S^{\Delta_1}(D)(L)$ and let $q_1 = \st{x_0^1, x_0^2, x_0^3, \cdots, x_0^e} \in S^{\Delta_1}(D)(L)$.  Then $Z_{p_1}-Z_{p_0} = Y_{q_1}-Y_{q_0} =a_L \in \chow(X_L)$.  To conclude the proof we use Bertini's theorem (Theorem \ref{T:MumBertini})  to obtain a geometrically integral smooth curve $C$ in $S^{\Delta_1}(D)$ passing through $p_1$ and $p_0$.
\end{proof}

\subsubsection{Proof of  $\eqref{T:main(ii)}\implies \eqref{T:main(iii)}$ in Theorem \ref{T:main}} \label{S:ii=>iii}
In view of Proposition \ref{P:algcurveST}, it suffices to prove\,:

\begin{lem}\label{L:taucurve}
Let $X$ be a  scheme of finite type over $K$, let $L/K$ be an algebraic
extension, and let
$a \in \operatorname{CH}_i(X_L)$
be a cycle class.  Suppose there exist a finite extension $L'/L$,  a smooth 
geometrically integral curve $C/K$, a cycle class 
$Z\in \operatorname{CH}_{i+1}(C\cross_K X)$ and $L'$-points $p_1, p_0\in C(L')$
such that  
$Z_{p_1}-Z_{p_0}=a_{L'}$ in $\operatorname{CH}_i(X_{L'})$.  

Then there exist a non-zero integer $N$, a smooth 
geometrically integral projective curve $C'/K$, a cycle class 
$Z'\in \operatorname{CH}_{i+1}(C'\cross_K X)$ and $L$-points $q_1, q_0\in C(L)$
such that  
$Z'_{q_1}-Z'_{q_0}=Na$ in $\operatorname{CH}_i(X_{L})$.    
\end{lem}
\begin{proof} Denote $\operatorname{pr} : X_{L'} \to X_L$ the natural
projection, and denote $[p_i]$ the push-forward of the fundamental class of
$p_i$, $i=0,1$, to $C$. 
By definition of the fiber of a cycle above a zero-cycle (see Section
\ref{S:algnum}), we have that $\operatorname{pr}_*(Z_{p_1} - Z_{p_0}) =
Z_{[p_1]} - Z_{[p_0]}$. On the other hand, $\operatorname{pr}_*(a_{L'}) =
[L':L]a$ by the projection formula. Applying a variant of Lemma \ref{L:algnum}
with $L$-points rather than $K$-points yields that $[L':L]a$ is  $\rho_{L/K}$-algebraically
trivial, and parameterized by a smooth quasi-projective 
 geometrically integral
scheme $T$. By Bertini's theorem (Theorem \ref{T:MumBertini}), we may take $T$ to be a  geometrically integral curve.  In fact, if  $K$ is perfect, we may take $T$ be be a geometrically integral projective curve.  If $K$ is imperfect of characteristic $p$, then after replacing $[L':L]a$ by $p^n[L':L]a$ for some $n$, we may use Lemma \ref{L:ProjCurve} to arrive at the same conclusion.
\end{proof}

This completes the proof of the main Theorem \ref{T:main}.

\subsection{Singular parameter spaces}  In this last section, we discuss a notion of algebraic triviality where one allows for singular parameter spaces.  
In Section \ref{S:Flat} we discussed an equivalent definition  of
algebraic  triviality  based on flat families of cycles and scheme
theoretic fibers.  There we required that the parameter space for the
cycles be smooth, which was necessary in light of Example \ref{E:RNC}
(see Remark \ref{R:PFSTconv}, below).   In that example, however,
the cycle causing the conflict was  evidently both $\sigma_{K/K}$- and
$\tau_{K/K}$-algebraically trivial.  Here we show that when the ground
field $K$ is perfect, this is always the case (Proposition
\ref{P:flatSingTriv}).

\begin{dfn} \label{D:flatlytrivial}
Let $X$ be a scheme of finite type over a field $K$. A cycle $\alpha \in \operatorname{Z}_i(X)$ is said to be \emph{flatly trivial} if there exists a (not necessarily smooth) integral scheme $T$ of finite type over $K$,   an effective  cycle $Z\in \operatorname{Z}_{i+d_T}( T\times_K X)$
flat over $T$ (i.e., there are integral subschemes $V_j\subseteq T\times_K X$
flat over $T$, $1\le j\le r$,  for some $r$, and  
$Z=\sum_{i=1}^r[V_i]$),
and $K$-points $t_1,t_0\in T(K)$ such that $Z_{t_1}-Z_{t_0}=\alpha$ in
$\operatorname{Z}_i(X)$. A cycle class is flatly trivial if it is the class of a cycle that is flatly trivial.
\end{dfn}

\begin{rem}
An argument similar to that in Lemma \ref{L:CycleEquiv} shows that if a cycle class $a\in \operatorname{CH}_i(X)$ is flatly trivial, then for any cycle $\alpha$ such that $a=[\alpha]$, we have that $\alpha$ is flatly trivial. 
\end{rem}

\begin{pro}\label{P:flatSingTriv}
Let $X$ be a scheme of finite type over a field $K$, and let $a \in \operatorname{CH}_i(X)$ be a cycle class. 
\begin{enumerate}
\item If $a$ is algebraically trivial, then $a$ is flatly trivial.

\item Assume $K$ is perfect. If $a$ is flatly trivial, then $a$ is $\sigma_{K/K}$-algebraically trivial (and therefore also $\tau_{K/K}$-algebraically trivial).
\end{enumerate}
\end{pro}

\begin{proof} Item (1) follows immediately from Lemma \ref{L:CycleEquiv}. As for item (2), consider a cycle $\alpha$ representing $a$ together with $T$, $Z$, $t_1$ and $t_0$ as in Definition \ref{D:flatlytrivial}. Since we are assuming that $K$ is perfect, there is by \cite{deJong} an alteration $\widetilde T \to T$ with $\widetilde T \to K$ smooth. Let $\tilde t_1$ and $\tilde t_0$ be any two $\bar K$-points in $\widetilde T$ lifting the points $t_1,t_0 \in T(K)$, and let $\widetilde Z$ be the flat effective cycle on $\widetilde T \times_K X$ that is the pre-image of the cycle $Z$. Then it is clear that $a_{\bar K} = \widetilde Z_{\tilde t_1} - \widetilde Z_{\tilde t_0}$ in $\operatorname{CH}_i(X_{\bar K})$, i.e., $a$ is $\sigma_{K/K}$-trivial.
\end{proof}

\begin{rem}\label{R:PFSTconv}
The converse of Proposition \ref{P:flatSingTriv} (1)  does not hold in general, 
as is shown in Example~\ref{E:RNC}.   Indeed, the cycle class
$[N]-[P]$ is flatly trivial, via the diagonal in the product of the
curve with itself.  However, the cycle class is not algebraically
trivial, as shown in the example.  Note that $[N]-[P]$ is clearly both
$\sigma_{K/K}$- and $\tau_{K/K}$-algebraically trivial, agreeing with
the assertion of  Proposition \ref{P:flatSingTriv}.  
\end{rem}

\appendix

\section{A review of Bertini theorems} \label{S:Bertini}

We start by recalling a consequence of the Bertini theorems and Chow's
theorem.  For algebraically closed fields, this is
\cite[Lem.~p.56]{mumfordAV}. It has long been understood how to pass
to arbitrary infinite fields\,; by \cite{poonen, charlespoonen}, the results
are now known to hold over an arbitrary field.  In particular, even
over a finite field, one knows that every geometrically irreducible
separated scheme of finite type contains a geometrically integral
curve.  In this paper, we use the case where the ambient scheme is
smooth, which follows from \cite{poonen}\,; for the sake of
completeness, we include a more general
statement that uses \cite{charlespoonen}.

\begin{teo}[Bertini--Charles--Poonen]\label{T:MumBertini}
Let $X$ be an irreducible  separated scheme of $\dim \ge 1$ of finite type over
a  field $K$, and let  $\bar x_1,\ldots, \bar x_n\in X(\bar K)$ be $\bar
K$-points of $X$.
There is an integral curve $C\subseteq X$ defined over $K$ containing the image
of   $\bar x_1,\ldots,\bar x_n$.    Moreover\,:

\begin{enumerate}
\item There exists such a $C$ so that there is a bijection between the
geometric components of $C$ and the geometric components of $X$, with the
property that  every component of $X_{\bar K}$ contains exactly one component of
$C_{\bar K}$.

\item If all of the points $\bar x_1,\ldots, \bar x_n$, viewed as $\bar
K$-points of $X_{\bar K}$, lie on the same component of $X_{\bar K}$, there
exists such a  $C$ such that all of the points 
$\bar x_1,\ldots, \bar x_n$, viewed as $\bar K$-points of $C_{\bar K}$, lie on
the same component of $C_{\bar K}$.

\item If $X$ is quasi-projective and $\dim X\ge 2$, there exists such a $C$
that is a general linear section of $X$ for some embedding of $X$ in a
projective space.  

\item If $X$ is quasi-projective and smooth over $K$, there exists such a  $C$
that is  smooth over $K$, and we may further specify the tangent directions of
$C$ inside $X$ at the points $\bar x_1,\ldots,\bar x_n$.
\end{enumerate}
\end{teo}

\begin{proof}
The case where $K=\bar K$ is  \cite[Lem.~p.56]{mumfordAV}.
The case where $X$ is assumed to be geometrically integral is
\cite[Thm.~1.1]{poonen} and \cite[Cor.~1.9]{charlespoonen}.  The general case is
essentially the same, but since the statements above require combining a few
results, we include the brief proof.  

We may always work with the reduced scheme structure, to assume that $X$ is
integral.  
From Nagata's compactification theorem (e.g., \cite{conradNagata})   $X$ admits
an open immersion into a proper scheme $X'$ over $K$.  
We may assume that $X'$ is integral (for instance, by taking the reduced scheme
structure).  
The  geometric components of $X'$ are in
bijection with the geometric components of $X$.  
By  Chow's theorem we may   dominate $X'$ with a projective scheme $X''$,
birational to $X'$.   We can assume that $X''$ is integral (for instance, by
taking reductions and components).   Again, since $X''$ is birational to $X'$,
the geometric components of $X'$ and $X''$ are in bijection.

Now lift the $\bar K$-points $\bar x_1,\ldots, \bar x_n$ of $X$ to $\bar
K$-points $\bar x_1'',\ldots ,\bar x_n''$ of $X''$, with $\bar x_i''$
($i=1,\ldots, n$)  lying on the component of  $X''_{\bar K}$ corresponding to
the component of $X_{\bar K}$ on which $\bar x_i$  lies.  
Using the Bertini irreducibility theorems  
(\cite[Thm.~6.3(4)]{jouanolou}, \cite[Thm.~1.2 and Rem.~1.2(d)]{charlespoonen})
we may find $C''\subseteq X''$ an integral curve that is a general linear
section of $X''$ for some projective embedding of $X''$, with $C''$ passing
through the images of the points $\bar x_1'',\ldots, \bar x_n''$.
Moreover, we may do this in such a way that there is a bijection between the
geometric components of $C''$ and the geometric components of $X''$, with the
property that  every component of $X''_{\bar K}$ contains exactly one component
of $C''_{\bar K}$.
Taking $C$ to be the restriction to $X$ of the image of $C''$ in $X'$ provides
the desired integral curve, and also establishes (1) and (2).  

Finally, if $X$ is quasi-projective, then $X''$ above can be taken to be a
projective completion of~$X$.  Then (3) follows from the argument above, and (4)
follows from the Bertini  theorems (\cite[Thm.~1.1, Thm.~1.2]{poonen} for the
case of finite fields).
\end{proof}

If $X$ is not assumed to be separated, the theorem can fail\,:

\begin{exa}\label{E:RMumInter}
Given a scheme $X$ of finite type over a field $K$, and $\bar K$-points $\bar
x_1,\bar x_2\in X(\bar K)$, if $X$ is not separated over $K$ there may be no
integral curve $C\subseteq X$ containing the image $\bar x_1$ and $\bar x_2$.  
Indeed, take $X$ to be the non-separated scheme obtained from two copies of
$\mathbb A^1_K$ by identifying all corresponding points but the origin, and take
$\bar x_1$ and $\bar x_2$ to be the two copies of the origin. 
\end{exa}

However, even if $X$ is not assumed to be separated,  the theorem does still 
imply the existence of chains of curves interpolating between points\,:

\begin{cor}\label{C:MumBertini}
Let $X$ be an irreducible scheme of finite type over a  field $K$ of $\dim \ge
1$, and let  $\bar x_1,\bar x_0\in X(\bar K)$ be $\bar K$-points of $X$.
There is a chain  of integral curves in $X$ containing the images of   $\bar
x_1,\bar x_0$.   More precisely, 
there are integral  curves $C_1,\ldots, C_r\subseteq X$ defined over $K$ for
some $r$,  so that for each $i=1,\ldots, r$   
there are $\bar K$-points $\bar x_{i_1},\bar x_{i_0}\in C_i(\bar K)$, with 
$\bar x_1=\bar x_{1_1}$, $\bar x_0=\bar x_{r_0}$ and $\bar x_{i_0}=\bar
x_{(i+1)_1}$ for $i=1,\dots,r-1$. 
Moreover\,:

\begin{enumerate}
\item There exists such a chain $C_1,\ldots,C_r$ so that for each $j$, the
geometric components of  $C_j$ are in bijection with the geometric components of
$X$, and  every component of $X_{\bar K}$ contains exactly one component of
$(C_j)_{\bar K}$ for $j=1,\ldots,r$.

\item If the points $\bar x_1,\bar x_0$, viewed as $\bar K$-points of $X_{\bar
K}$, lie on the same component of $X_{\bar K}$, there exists such a chain 
$C_1,\ldots, C_r$ such that the points 
$\bar x_1,\bar x_0$, viewed as $\bar K$-points of $(C_1)_{\bar K}$ and
$(C_r)_{\bar K}$, lie on the same connected component of the union $C_{\bar
K}=\bigcup_{i=1}^r(C_i)_{\bar K}$.

\item If $X$ is smooth over $K$ and $\dim X\ge 2$, there exists such a chain 
$C_1,\ldots,C_r$ such that the union $C=\bigcup_{i=1}^rC_i$ is a nodal
curve, and such that the image of each point $\bar x_2, \ldots, \bar
x_{r-1}$ has finite separable residue field over $K$.

\end{enumerate}
\end{cor}

\begin{proof}
Since $X$ is of finite type over $K$, we may cover $X$ with finitely
many quasi-projective open subsets $U_1,\ldots, U_r$.
Label these so that $\bar x_1\in U_1$ and $\bar x_0\in U_r$.  
Set $\bar x_{1_1}=\bar x_1$.   We label new points inductively.  If $r=1$,
then set $\bar x_{1_0}=\bar x_0$. 
Otherwise, if $r>1$, 
to avoid indexing trivialities, let $\bar x_{1_0}\in U_1\cap U_2$ be a $\bar
K$-point of $U_1$ that is different from $\bar x_{1_1}$.  Set $x_{2_1}=x_{1_0}$.
Continue inductively defining points in this way until we have $\bar
x_{r_0}=\bar x_0\in U_r$.  
Now we may use Theorem \ref{T:MumBertini} to find curves  $C_i\subseteq U_i$ passing
through the given points, and with the asserted properties.  (For the
final claim of (3), use the fact that closed points with residue field
finite and separable over the base field are Zariski dense in any
geometrically reduced scheme which is locally of finite type.)
\end{proof}

\section{Weil's argument}\label{S:Weil}

For comparison with Proposition \ref{P:algab}, we include Weil's argument that over an algebraically closed field, curve triviality and abelian variety triviality agree.  In fact,  with Definitions \ref{D:sigmatau} and~\ref{D:TrivCycST} in mind, Weil's argument shows that $\rho_{L/K}$-curve triviality implies $\sigma_{L/K}$-abelian variety triviality for perfect fields $K$.  The key point is that the Brill--Noether theory utilized in his argument  requires the existence of divisors over the base field, which cannot in general be assumed to exist.

\begin{lem}[{\cite[p.108]{weil54}, \cite[p.60]{Lang}}]
\label{L:LangKgeom}
Let $X$ be a  scheme of finite type over $K$, let $L/K$ be an algebraic
extension, and let
$a \in \operatorname{CH}_i(X_L)$
be a cycle class.  Suppose there exist  a smooth 
geometrically integral projective curve $C/K$ of genus $g$, a cycle class 
$Z'\in \operatorname{CH}_{i+1}(C\cross_K X)$ and $L$-points $p_1, p_0\in C(L)$
such that  
$Z'_{p_1}-Z'_{p_0}=a$ in $\operatorname{CH}_i(X_L)$.  

Then there exist  a cycle class $Z\in
\operatorname{CH}_{i+g}(\operatorname{Pic}^g_{C/K}\times_K X)$, a finite separable
field extension $L'/L$, and $L'$-points $t_1, t_0\in
\operatorname{Pic}^g_{C/K}(L')$ such that $Z_{t_1} - Z_{t_0} =a_{L'}$ in
$\operatorname{CH}_i(X_{L'})$.  
\end{lem}

\begin{rem}
If  $C$ admits a $K$-point (or more generally a line bundle of degree $g$ over
$K$), then $\operatorname{Pic}^g_{C/K}$ is isomorphic to an abelian variety,
namely $\operatorname{Pic}^0_{C/K}$.  
\end{rem}

\begin{proof}
If $g=0$, then $a=0\in \operatorname{CH}_i(X_L)$,
and there is nothing to do, as we can for instance take $Z$ to be the zero cycle
class.   
Otherwise, consider the $g$-fold symmetric
product $S^g C$ together with the induced cycle $S^g Z'$ on
$S^g  C \cross_K X$, 
and the
Abel--Jacobi map $\aj: S^g C \to \pic^g_{C/K}$ (e.g.,
\cite[Def.~9.4.6, Rem.~9.3.9]{kleimanPIC}).

Pulling back to $\bar K$, Brill--Noether theory provides the following two
facts.  First,  the Abel--Jacobi map 
$\bar{\aj}:S^g\bar C \to \pic^g_{\bar C/\kbar}$ is proper and birational, and
the image
of the exceptional locus is a 
set of dimension $g-2$ (this set is identified with the Serre dual of the
Brill--Noether locus $W_{g-2}(\bar C)$). Second, denoting by $\bar p_1,\bar p_0$
the $\bar K$-points associated to $p_1,p_0$, then  if $\bar p_2, \ldots, \bar
p_g \in
C(\kbar)$ are sufficiently general (and therefore can be taken to be defined
over a finite separable extension of $L$), the images $\bar t_1$ and $\bar t_0$
of
$\bar p_1$ and $\bar p_0$ respectively under the map
\[
\xymatrix@R=.3em{
\bar C \ar@{^(->}[r] & \pic^g_{\bar C/ \kbar} \\
\bar p \ar@{|->}[r] & \bar p+(\bar p_2+ \cdots + \bar p_g)
}
\]
avoid the exceptional locus of $\bar{\aj}$.
Now define a cycle 
$Z :=  (\aj\cross \operatorname{id}_X)_* S^g Z'$ on $\pic^g_C
\cross_K X$.  Then 
$$
Z_{\bar t_1}-Z_{\bar t_0} = (Z'_{\bar p_1}+Z'_{\bar p_2}+\cdots +Z'_{\bar
p_g})-(Z'_{\bar p_0}+Z'_{\bar p_2}+\cdots
+Z'_{\bar p_g})=
Z'_{\bar p_1}-Z'_{\bar p_0} = a_{\bar K}.
$$
Letting $L'/L$ be the finite separable extension over which the $p_i$ are
defined,  we are done.  
\end{proof}

\begin{rem}
\label{R:Qcurveabvar}
It is not hard to write down examples where Weil's construction in the proof above 
necessitates a field extension. For instance, let $C/\ff_3$ be the
  hyperelliptic curve of genus $2$ with affine model $y^2 =
  x^5-x^2-1$.  Then $C(\ff_3)= \st{P,Q}$, where $P = [0:1:0]$ and $Q =
  [2:0:1]$, and $S^2(C)(\ff_3) = \st{2[P],2[Q],[P]+[Q]}$.  Since
  $h^0(2P) = h^0(2Q) = 2>1$, there is no effective $\ff_3$-rational
  divisor $D$ of degree $g-1=1$
  such that $P+D$ and $Q+D$ are both Brill--Noether general.
\end{rem}


\bibliographystyle{amsalpha}
\providecommand{\bysame}{\leavevmode\hbox to3em{\hrulefill}\thinspace}
\providecommand{\MR}{\relax\ifhmode\unskip\space\fi MR }
\providecommand{\MRhref}[2]{%
	\href{http://www.ams.org/mathscinet-getitem?mr=#1}{#2}
}
\providecommand{\href}[2]{#2}

\end{document}